\documentclass[11pt,oneside,english]{amsart}
\usepackage{subfiles, mathtools}
\usepackage{amssymb}
\usepackage[colorlinks]{hyperref}
\usepackage{graphicx}
\usepackage[pagewise]{lineno}
\usepackage{ulem}

\makeatletter \theoremstyle{plain}
 \newtheorem{thm}{Theorem}[section]
 \newtheorem{lem}[thm]{Lemma}
 \newtheorem{prop}[thm]{Proposition}
 
 \numberwithin{equation}{section} 
 \numberwithin{figure}{section} 
 \theoremstyle{plain}
 \theoremstyle{definition}
 \newtheorem{defn}[thm]{Definition}
 \newtheorem{rem}[thm]{Remark}
  \newtheorem{cor}[thm]{Corollary}
  \newtheorem{ex}[thm]{Example}

\newcommand{\bH}{{{\bf H}}}

\newcommand{\Aa}{{\mathcal {AA}}}

\newcommand{\calH}{{{\mathcal H}}}

\newcommand{\C}{{{\mathbb C}}}

\newcommand{\R}{{{\mathbb R}}}

\newcommand{\mH}{{{\mathbb H}}}

\usepackage{babel}

\makeatother

\newcommand\blfootnote[1]{%
  \begingroup
  \renewcommand\thefootnote{}\footnote{#1}%
  \addtocounter{footnote}{-1}%
  \endgroup
}

\providecommand{\subjclass}[1]
{
  \small	
  {\textit{$2020$ Mathematics Subject Classifications.}} #1
}

\begin{document}

\title [Horizontal curvatures]{Horizontal curvatures of surfaces in 3D contact sub-Riemannian Lie groups}

\author{Elia Bubani \& Andrea Pinamonti \& Ioannis D. Platis \& Dimitrios Tsolis}

\newcommand{\Addresses}{{
  \bigskip
  \footnotesize

  E.~Bubani,  \textsc{University of Fribourg, Department of Mathematics, 
Ch. du Musée 23,
1700 Fribourg,
Switzerland.}\par\nopagebreak
  \textit{E-mail address}, E.~Bubani: \texttt{elia.bubani@unifr.ch}

  \medskip

  A.~Pinamonti (Corresponding author), \textsc{Department of Mathematics, University of Trento, via Sommarive 14, 38123 Povo, Italy}\par\nopagebreak
  \textit{E-mail address}, A.~Pinamonti: \texttt{andrea.pinamonti@unitn.it}

  \medskip

  I.D.~Platis, \textsc{Department of Mathematics, University of Patras, Panepistimioupolis, 26504 Rion, Achaia, Greece}\par\nopagebreak
  \textit{E-mail address}, I.D.~Platis: \texttt{idplatis@upatras.gr}

  \medskip
  D.~Tsolis, \textsc{Department of Mathematics, University of Patras, Panepistimioupolis, 26504 Rion, Achaia, Greece}\par\nopagebreak
  \textit{E-mail address}, D.~Tsolis: \texttt{d.tsolis@upatras.gr}

}}

\date{\today}

\begin{abstract}
In this paper we study horizontal curvatures for surfaces embedded in three-dimensional contact sub-Riemannian Lie groups. Using a Riemannian approximation scheme, we derive explicit formulas for horizontal Gauss curvature, horizontal measn curvature, and symplectic distortion for surfaces embedded in three dimensional Lie groups with a sub-Riemannian structure obtained by a contact form. We focus on two primary examples: the Heisenberg group and the affine-additive group. We classify surfaces of revolution within these groups that exhibit constant horizontal curvatures, often expressing their profiles through elementary or elliptic integrals. 
\end{abstract}
\maketitle

\blfootnote{\subjclass{53C17, 53D10}}
\blfootnote{\thanks{The first author was supported by the Swiss National Science Foundation, Grants Nr. 191978, 10001161 and  204501. The second author is a member of the Istituto Nazionale di Alta Matematica (INdAM), Gruppo Nazionale per l'Analisi Matematica, la Probabilità e le loro Applicazioni (GNAMPA), and is supported by the University of Trento, the MIUR-PRIN 2022 Project \textit{Regularity problems in sub-Riemannian structures}  Project code: 2022F4F2LH and the INdAM-GNAMPA 2025 Project \textit{Structure of sub-Riemannian hypersurfaces in Heisenberg groups}, CUP ES324001950001. The third author was supported by the Medicus program, Grant Nr. 83765.}}

\tableofcontents

\section{Introduction}

The study of surfaces and their various curvatures lies at the heart of differential geometry, with profound implications in analysis, geometric measure theory, and mathematical physics. In the general Riemannian setting, classical notions of mean and Gaussian curvature are well-established and provide deep insight into the local and global behaviour of embedded hypersurfaces. However, many contemporary problems in geometric analysis and sub-elliptic PDE arise in settings where the underlying geometry is not Riemannian but sub-Riemannian; here, the ambient metric is inherently anisotropic.

Sub-Riemannian spaces model constrained motion: admissible curves are required to be tangent to a distinguished, typically non-integrable, \textit{horizontal} distribution. Such geometries arise naturally in numerous areas inside mathematics and natural sciences in general. A particularly rich and flexible class is provided by \textit{contact sub-Riemannian manifolds}, where the horizontal distribution is realised as the kernel of a contact form and is complemented by the Reeb vector field. Among them, \textit{three-dimensional contact sub-Riemannian Lie groups} are exceptional, as they are sufficiently structured to allow explicit computations and general enough to capture many phenomena that appear in higher dimensions.

A fundamental challenge in this setting is to develop intrinsic notions of curvature for embedded surfaces that are genuinely adapted to the sub-Riemannian structure. Classical Riemannian definitions depend on the full ambient metric and the Levi-Civita connection (see \cite{M-76}); neither of which is available when the metric is defined only on the horizontal distribution. In particular, surface geometry must account for the presence of \textit{characteristic points}, where the tangent plane aligns with the horizontal distribution and the horizontal normal degenerates. Therefore, a satisfactory definition of curvature should: a) coincide with the Riemannian one whenever the sub-Riemannian structure is induced by a Riemannian metric,b)  be invariant under the natural isometry group of the sub-Riemannian geometry (e.g.\ left translations), c) encode the non-holonomic constraints imposed by the horizontal distribution and d) reduce to the known expressions in model spaces such as the Heisenberg group.

In this paper we provide a systematic treatment of curvature for surfaces embedded in three-dimensional contact sub-Riemannian Lie groups. We focus on two characteristic examples: the Heisenberg group and the affine-additive group. Each carries a left-invariant contact structure given by a contact form, its Reeb vector field, and a metric on the horizontal bundle. Our approach is based on a \textit{Riemannian approximation scheme}: we introduce a one-parameter family of Riemannian metrics that converges to the sub-Riemannian metric in the Gromov-Hausdorff sense (see Section \ref{sec:Riem-approx}). Within this family we employ Cartan's method of moving frames: for an embedded surface, we compute the Riemannian second fundamental form and the ambient Gauss curvatures in an adapted orthonormal frame. A careful asymptotic analysis of these quantities as the approximation parameter tends to zero produces natural sub-Riemannian limits and leads to intrinsic definitions of the \textit{horizontal mean curvature} $H_{\Sigma}^h$ and the \textit{horizontal Gauss curvature} $K_{\Sigma}^h$ of a surface $\Sigma$, away from its characteristic set (where the horizontal gradient vanishes), see Section \ref{sec:not-curv}.

After providing the proper definitions, we focus in obtaining explicit formulae that are amenable to applications and classification problems. Together with $H_{\Sigma}^h$ and $K_{\Sigma}^h$ we also analyse the \textit{symplectic distortion} $Q_{\Sigma}^h$, which captures the way the surface geometry is twisted relative to the ambient contact structure and plays an important role in the geometric identities emerging from the approximation scheme. We apply these formulae to make explicit computations for several natural classes of surfaces. Moreover, we obtain classification results for surfaces of revolution with constant horizontal Gauss curvature, constant horizontal mean curvature, and constant symplectic distortion, deriving explicit formulae.

Our results connect with, and extend, several earlier contributions on curvature in sub-Riemannian geometries. In the Heisenberg group, notions of sub-Riemannian Gaussian curvature and related Gauss-Bonnet type identities have been developed in \cite{BTV,DV-GaussBonnet}. Further extensions to contact manifolds and general sub-Riemannian settings appear in \cite{V-GBgeneral,GHVM-GaussBonnet,BBP-GaussBonnet}. The present definition of horizontal mean curvature is compatible with known formulae in the Heisenberg group (see \cite{CDPT}). Moreover, inspired by classical investigations in Euclidean space (e.g.\ \cite{Massey}) and by recent studies of rotationally invariant surfaces and constant-curvature problems in Heisenberg groups (see \cite{P-SRS,RR-Rot.inv.,Vel-RSCC1,Vel-RSCC2}), the framework developed here provides a blueprint for analogous questions in other contact sub-Riemannian geometries. The explicit computations and classification results we obtain are intended to support further investigations on constant curvature surfaces and isoperimetric problems.

The paper is organised as follows. In Section \ref{sec-prel} we introduce the general framework of three-dimensional contact sub-Riemannian Lie groups and fix notation. Section \ref{sec-notions} is devoted to the Riemannian approximation scheme: we construct adapted moving frames and coframes along an embedded surface and, using Cartan's structural equations, derive explicit formulae for the connection forms, the relevant sectional curvatures, and the limiting expressions that define the horizontal mean and Gauss curvatures, together with the symplectic distortion (see Section \ref{sec:not-curv}). Section \ref{sec-appl} contains applications to the two model groups, namely the Heisenberg group (Section \ref{sec:ex-Heis}) and the affine-additive group (Section \ref{Sec:ex-AA}).

\medskip

\noindent{\textbf{Acknowledgements}} The authors would like to thank Zolt\'an M. Balogh (Bern, Switzerland) for fruitful discussions on the subject of this paper. We also thank Georgios Simantiras (Patras, Greece) for interesting conversations about Example \ref{cc-flask}.

\color{black}

\section{Preliminaries}\label{sec-prel}

\subsection{Three dimensional contact sub-Riemannian Lie groups}
We consider an arbitrary  3-dimensional connected Lie group $G$; we denote by $e$ the identity element of $G$ and also, for $g\in G$ we denote by $L_g:G\to G$ the left translation given by $L_g(h)=gh,\, h\in G$. We shall assume that $G$ is equipped with a contact form $\vartheta$, i.e., a $1$-form such that $\vartheta\wedge d\vartheta\neq0$. The horizontal distribution $\mathcal{H}$ in the tangent bundle ${\mathrm T}\,G$ is given by $\mathcal{H}:=\ker\vartheta={\rm span}\{X,Y\}$, where $X$ and $Y$ are left-invariant vector fields. Since $\vartheta$ is a contact form, the distribution $\mathcal{H}$ is completely non-integrable: ${\mathrm T}\,G=\calH+[\calH,\calH]$. We denote by $T$ the Reeb field, i.e., the left-invariant vector field such that $\vartheta(T)=1$ and $d\vartheta(T,\cdot)=0$. A frame for ${\mathrm T}\,G$ is given by $\{X,Y,T\}$ and we assign a coframe $\{\omega_1,\omega_2,\vartheta\}$ given by the dual 1-forms to $\{X,Y,T\}$. We endow $G$ with a CR structure $J:\mathcal{H}\to \mathcal{H}$  which we extend to $J:\mathrm{T}\,G\to\mathrm{T}\,G$ by setting
$$
JX=Y,\quad JY=-X, \quad JT=0.
$$
A left-invariant metric is constructed in $G$ considering the left-invariant sub-Riemannian metric $g:\mathcal{H}\times\mathcal{H}\to \R$ such that $\{X,Y\}$ is an orthonormal frame. 
The distance associated  to $g$ is defined as follows:
an absolutely continuous curve $\gamma:[a,b]\to G$ shall be called {\it horizontal} if $\dot{\gamma}(s)\in\calH_{\gamma(s)}$ for almost every $s\in[a,b]$. Then, the {\it horizontal velocity} $|\dot\gamma(s)|$ of $\gamma$ is defined by
$$|\dot\gamma(s)|=\sqrt{g(\dot\gamma(s),X_{\gamma(s)})^2+g(\dot\gamma(s),Y_{\gamma(s)})^2}\,$$
and the {\it horizontal length} $\ell(\gamma)$ of $\gamma$ is given by $$\ell(\gamma)=\int_a^b|\dot\gamma(s)|\,ds.$$
Now, for every $p, q  \in G$, the {\it Carnot-Carath\'eodory or sub-Riemannian distance} $d_{CC}$ associated with the metric $g(\cdot,\cdot)$ is defined by:  
\begin{equation*}\label{CC distance}
    d_{CC}(p,q)=\inf_{\gamma\in \Gamma_{(p,q)}}\{\ell(\gamma) \},
\end{equation*}
where $\Gamma_{(p,q)}=\{\gamma,\gamma:[0,1]\to G \text{ horizontal and } \gamma(0)=p,\, \gamma(1)=q\}$. 
The definition \eqref{CC distance} depends only on the values of $g(\cdot,\cdot)$ in $\mathcal{H}\times \mathcal{H}$.
Moreover, since $\mathcal{H}$ is completely non-integrable, the distance $d_G$ is finite, geodesic, and
induces the manifold topology (cf. \cite{Mont}).
for any 1-form $\omega$ and vector fields $X,Y$.
The following proposition gives explicit relations between the vector fields $X,Y$ and $T$:
\begin{prop}\label{Prop: abc}
Let $X,Y,T$ be as above. Then there exist constants $a_i,b_i$, $i\in\{1,2,3\}$ and $c\neq 0$ such that
$$
[X,T]=a_1X+b_1Y, \quad [Y,T]=a_2X+b_2Y, \quad [X,Y]=a_3X+b_3Y+cT.
$$ 
Moreover, $a_i,b_i$, $i\in\{1,2,3\}$ and $c$ satisfy
\begin{equation}\label{eq: liebr-const}
   \left|\begin{matrix}
    a_1&b_1\\
    a_3&b_3
\end{matrix}\right|=
\left|\begin{matrix}
    a_2&b_2\\
    a_3&b_3
\end{matrix}\right|=0,\quad a_1+b_2=0. 
\end{equation}
\end{prop}
\begin{proof}
We first prove that $ [X,T],[Y,T]\in\ker\vartheta.$
There are smooth functions $a_1,b_1,c_1:G\to \R$ such that
$$
  [X,T]=a_1X+b_1Y+c_1T.
$$
Since $T$ is the Reeb field, we have by definition $ \vartheta(X)=0$ and $ d\vartheta(T,\cdot)=0$.
Thus
$$
  0=d\vartheta(X,T)=X(\vartheta(T))-T(\vartheta(X))-\vartheta([X,T])=-\vartheta([X,T])=-c_1,
$$
where the second identity follows from the formula \begin{eqnarray}\label{domegaXY}
    d\omega (X,Y)&=& X(\omega (Y))-Y(\omega(X))-\omega([X,Y]),
\end{eqnarray}
 for every $\omega$  a smooth 1-form and every $X,Y$ smooth vector fields defined on a smoth manifold (cf. Proposition 14.29 in \cite{Lee-Intro}).
Therefore $[X,T]=a_1X+b_1Y\in\ker\vartheta$ and the proof for $[Y,T]$ is similar.

Now, there are smooth functions $a_3,b_3,c: G \mapsto \R$ such that
$$
[X,Y] =a_3X+b_3Y+cT.
$$
Since  $\vartheta$ is a contact form,
$$
0\neq d\vartheta(X,Y)=-\vartheta([X,Y])=c,
$$
implying that $c\neq 0$. Thus the left invariance of $X,Y,T$ yields
\begin{eqnarray*}\label{eq: pushfwd [X,T]}
   (DL_g)_e[X,T]_e&=&a_1(e)(DL_g)_eX_e+b_1(e)(DL_g)_eY_e\nonumber\\
   &=&a_1(e)X_g+b_1(e)Y_g,
   \end{eqnarray*}
  and also
   \begin{eqnarray*}\label{eq: [X,T]1}
       [X,T]_g&=&a_1(g)X_g+b_1(g)Y_g.
   \end{eqnarray*}
   On the other hand, the left invariance of the Lie bracket $[X,T]$ gives
\begin{equation*}\label{eq: [X,T]2}
    [X,T]_g= (DL_g)_e[X,T]_e=a_1(e)X_g+b_1(e)Y_g,
\end{equation*}
and we deduce that $a_1,b_1$ are constant.
Now,
\begin{eqnarray*}
    [X,[Y,T]]&=&[X,\,a_2X+b_2Y]\\
    &=&b_2[X,Y]\\
    &=&b_2a_3\,X+b_2b_3\,Y+b_2c\,T.
\end{eqnarray*}
Similarly,
$$
[Y,[T,X]]=a_1a_3\,X+
    a_1b_3\,Y+a_1c\,T,
    $$
    and
    $$[T,[X,Y]]=(-a_1a_3-a_2b_3)\,X+(
    -a_3b_1-b_2b_3)\,Y.
$$
Thus, the Jacobi identity with respect to $X,Y$ and $T$ yields
\begin{eqnarray*}
    &&
    a_3b_2-a_2b_3=0,\\
    &&
    a_1b_3-a_3b_1=0,\\
    &&
    c(a_1+b_2)=0.
\end{eqnarray*}
Since $c\neq 0$, the proof is complete.
\end{proof}
\subsection{Sub-Riemannian isometries}
A homeomorphism $f: G \to G $ is said a {\it CC-isometry} if  
\begin{equation}\label{def: eq: CC-isom}
    d_{CC}(f(p),f(q))=d_{CC}(p,q),\quad\forall p,q\in G.
\end{equation}
We underline here that $CC$-isometries are smooth analytic maps (see Theorem 1.3 in \cite{DO-I}).
We denote by ${\rm Isom}_{CC}(G)$ (resp. by ${\rm AutIsom}_{CC}(G)$) the group of $CC$-isometries of $G$ (resp. the group of automorphisms of $G$ that are $CC$-isometries). The next theorem follows from Theorem 1.2,  \cite{KD-17}:
\begin{thm}\label{Thm:Isom}
     Let $G$ be a $3$-dimensional connected contact sub-Riemannian Lie group $G$ with groups of CC-isometries ${\rm Isom}_{CC}(G)$ and ${\rm AutIsom}_{CC}(G)$. If $G$ is nilpotent, then $${\rm Isom}_{CC}(G) = G \rtimes {\rm AutIsom}_{CC}(G)\,.$$
\end{thm}
Theorem \ref{Thm:Isom} holds only when $G$ is nilpotent, cf. \cite{KD-17}, Section 4.

\section{Notions of curvature}\label{sec-notions}
In Section \ref{sec:Riem-approx} we use a Riemannian approximation scheme  to provide notions of horizontal mean curvature and of horizontal Gaussian curvature for a $C^2$-smooth surface embedded in $G$ in Section \ref{sec:not-curv}. 
In this approach we use Cartan's formalism, we refer for details about Cartan's method to \cite{Chen-Lin-Biharmonic}, Chpt. 3. For a surface $\Sigma$ embedded in $G$ and by using Cartan's structural equations we derive formulae for the sectional curvature $\overline K^\epsilon$ of $G$ and for the second fundamental form $\mathbf{II}^\epsilon_\Sigma$ related to $\Sigma$. The formulae for the horizontal mean curvature $H^h_\Sigma$ and the horizontal Gaussian curvature $K^h_\Sigma$ will then arise by taking the limit case. We subsequently discuss the invariance of all those features in Section \ref{sec:curv-invariance}, horizontal geodesic curvature and horizontal normal curvature in Section \ref{sec:hor-geod} as well as $G$-cylindruical surfaces in Section \ref{sec:G-cyl}.
\subsection{The Riemannian approximation method}\label{sec:Riem-approx}
For $\epsilon>0$, we consider a family of Riemannian metrics $(g_\epsilon)_{\epsilon>0}$ on $G$ such that the frame $\{X,Y,T_\epsilon=\epsilon T\}$ is orthonormal and we denote its corresponding coframe by $\{\omega_1,\omega_2,\theta_\epsilon=(1/\epsilon)\theta\}$; the associated Riemannian distance to $g_\epsilon$ is denoted by $d_\epsilon$.
 Let $\omega_1=X^*$, $\omega_2=Y^*$ and $\vartheta_\epsilon=\vartheta/\epsilon=T_\epsilon^*$. The Lie bracket relations for $G$ are now written as:
\begin{equation*}\label{eq: bracket relations eps}
    [X,T_\epsilon]=\epsilon(a_1X+b_1Y), \quad
    [Y,T_\epsilon]=\epsilon(a_2X+b_2Y),\quad
    [X,Y]=a_3X+b_3Y+(c/\epsilon)T_\epsilon,
\end{equation*}
where $a_i,b_i$ and $c$ satisfy \eqref{eq: liebr-const} for all $i\in\{1,2\}$.
Applying formula \eqref{domegaXY} to \eqref{eq: bracket relations eps} we obtain
\begin{equation*}
    \begin{aligned}\label{eq: d old}
 d\omega_1=&-a_3\,\omega_1\wedge\omega_2-\epsilon \left(a_1\,\omega_1\wedge\theta_\epsilon+a_2\,\omega_2\wedge\vartheta_\epsilon\right),\\
 d\omega_2=&-b_3\,\omega_1\wedge\omega_2-\epsilon\left(b_1\,\omega_1\wedge\theta_\epsilon+b_2\,\omega_2\wedge\vartheta_\epsilon\right),\\
 d\vartheta_\epsilon=&-(c/\epsilon)\,\omega_1\wedge\omega_2.
\end{aligned}
\end{equation*}
From the first structural equations there exist three 1-forms 
$
\theta_1^2,\theta_1^3,$ and $\theta_2^3
$
such that
\begin{eqnarray*}
d\omega_1&=&\theta_1^2\wedge\omega_2+\theta_1^3\wedge\vartheta_\epsilon,\\
d\omega_2&=&-\theta_1^2\wedge\omega_1+\theta_2^3\wedge\vartheta_\epsilon,\\
d\theta_\epsilon&=&-\theta_1^3\wedge\omega_1-\theta_2^3\wedge\omega_2.
\end{eqnarray*}
We write
\begin{eqnarray*}
    \theta_1^2&=&(\theta_1^2)_1\,\omega_1+(\theta_1^2)_2\,\omega_2+(\theta_1^2)_3\,\vartheta_\epsilon,\\
    \theta_1^3&=&(\theta_1^3)_1\,\omega_1+(\theta_1^3)_2\,\omega_2+(\theta_1^3)_3\,\vartheta_\epsilon,\\
    \theta_2^3&=&(\theta_2^3)_1\,\omega_1+(\theta_2^3)_2\,\omega_2+(\theta_2^3)_3\,\vartheta_\epsilon.
\end{eqnarray*}
Then by equating we have:
\begin{eqnarray*}
   \theta_1^2&=&-a_3\,\omega_1-b_3\,\omega_2+\left(\frac{\epsilon}{2}(a_2-b_1)-\frac{c}{2\epsilon}\right)\,\vartheta_\epsilon,\\
   \theta_1^3&=&-\epsilon a_1\,\omega_1-\left(\frac{\epsilon}{2}(b_1+a_2)+\frac{c}{2\epsilon}\right)\,\omega_2,\\
   \theta_2^3&=&\left(\frac{c}{2\epsilon}-\frac{\epsilon}{2}(b_1+a_2)\right)\,\omega_1\epsilon b_2\,\omega_2.    \end{eqnarray*}
Now, the curvature 2-forms
$
\Phi^2_1, \Phi_3^1$ and $\Phi_2^3,
$
are given by the second structural
equations:
\begin{eqnarray*}
 \Phi_1^2&=&d\theta_1^2+\theta_1^3\wedge\theta_2^3,\\
 \Phi_1^3&=&d\theta_1^3-\theta_1^2\wedge\theta_2^3,\\
 \Phi_2^3&=&d\theta_2^3+\theta_1^2\wedge\theta_1^3.
\end{eqnarray*}
The sectional curvatures of the characteristic plane spanned by $X$ and $Y$ is subsequently given by
\begin{eqnarray}\label{SecCurvXY}
    K_\epsilon(X,Y)=-\Phi_1^2(X,Y)&=&-d\theta_1^2(X,Y)-\theta_1^3\wedge\theta_2^3(X,Y) \nonumber \\
    &=&-X\theta_1^2(Y)+Y\theta_1^2(X)+\theta_1^2([X,Y])-\theta_1^3(X)\theta_2^3(Y)+\theta_1^3(Y)\theta_2^3(X) \nonumber \\
    &=&-a_3^2-b_3^2+\frac{c}{2}(a_2-b_1)+\epsilon^2\left(\frac{(a_2+b_1)^2}{4}+ a_1^2 \right)-\frac{3c^2}{4\epsilon^2},
    \end{eqnarray}
    and in a similar manner we find that the sectional curvatures of the characteristic planes spanned by $X$ and $T$ as well as by $Y$ and $T$ are, respectively:
    \begin{eqnarray}
\label{SecCurvXT}
    K_\epsilon(X,T_\epsilon)=-\Phi^3_1(X,T_{\epsilon})&=& 
    \frac{c^2}{4\epsilon^2}+\frac{\epsilon^2}{4}\left((a_2+b_1)(a_2-3b_1)-4a_1^2\right)-\frac{c}{2}(a_2+b_1),\, \\
    \label{SecCurvYT} K_\epsilon(Y,T_\epsilon)=-\Phi^3_2(Y,T_{\epsilon})&=& 
    \frac{c^2}{4\epsilon^2}-\frac{\epsilon^2}{4}\left((a_2+b_1)(3a_2-b_1)+4a_1^2\right)+\frac{c}{2}(a_2+b_1).
    \end{eqnarray}
    The following holds (cf. \cite{Gromov}, Chpt. 1, or \cite{D-MLG}, Theorem 12.3.7):
\begin{thm}
The family of metric spaces $(G,d_\epsilon)$ converges to the metric space $(G,d_{CC})$ in the
pointed Gromov–-Hausdorff sense as $\epsilon\to 0^+$.
\end{thm}
We now fix once for all the totality of our assumptions on the surface $\Sigma$. Let $u:G\to\R$ be a $C^2$-smooth function such that its tangent map $Du$ is surjective for all $p\in u^{-1}(0)$. Then the regular level set theorem implies that the set 
\begin{equation*}
    \Sigma=\{p\in G:u(p)=0\},
\end{equation*}
is a regular (embedded) 2-dimensional submanifold of $G$.  We call such a submanifold a {\it hypersurface} or simply a {\it surface} in $G$ and since it is always orientable we fix an orientation for $\Sigma$. 
We say that a point $p \in\Sigma$ is \textit{characteristic} if
\begin{equation}\label{char point}
    \nabla_H u(p):=(Xu,Yu)_{|p}=(0,0),
\end{equation}
and the \textit{characteristic set} $\mathcal{C}(S)$ of $\Sigma$ comprises of its characteristic points:
\begin{equation*}
    \mathcal{C}(\Sigma)=\{p\in\Sigma:\nabla_H u(p)=(0,0)\}.
\end{equation*}
Following the notation used in  \cite{CDPT}, as well as  in \cite{BTV}, we set
$$
p=Xu,\quad q=Yu,\quad r=T_\epsilon u,
$$
and subsequently,
\begin{equation}
\!
    \begin{aligned}\label{notation MF 0}
    l=\|\nabla_H u\|:=\sqrt{(Xu)^2+(Yu)^2},\quad \overline{p}=\frac{p}{l},\quad\overline{q}=\frac{q}{l},\quad \overline{r}=\frac{r}{l}\\
    l_\epsilon=\sqrt{(Xu)^2+(Yu)^2+(T^\epsilon u)^2},\quad \overline{r_\epsilon}=\frac{r}{l_\epsilon},\\
    \overline{p_\epsilon}=\frac{p}{l_\epsilon}\,\quad\text{and}\quad \overline{q_\epsilon}=\frac{q}{l_\epsilon}.
\end{aligned}
\end{equation}
In particular $\overline{p}^2+\overline{q}^2=1$ and $\left(\frac{l}{l_\epsilon} \right)^2 (\overline{r}^2+1)=1$. From Definition \eqref{char point} it is straightforward that all functions in \eqref{notation MF 0} are well defined on $\Sigma\setminus\mathcal{C}(\Sigma)$.
The \textit{Riemannian gradient} $\nabla_\epsilon u$ of $u$ is given by
$$\nabla_\epsilon u=(Xu)X+(Yu)Y+(T_\epsilon u)T_\epsilon$$ and the \textit{Riemannian unit normal} $n_\Sigma$ to $\Sigma$ is
\begin{equation*}
    n_\Sigma=\frac{\nabla_\epsilon u}{\|\nabla_\epsilon u\|}.
\end{equation*}
\begin{defn}\label{E_1, E_2}
    An orthonormal frame $\{E_1,E_2,n_\Sigma\} $ associated to the surface $\Sigma$ is given by
\begin{equation*}
        E_1=Jn_\Sigma=-\overline{q}\,X+\overline{p}\,Y,\quad 
    E_2=\frac{l}{l_\epsilon}(\overline{r}\,\overline{p}\,X+\overline{r}\,\overline{q}\,Y-T_\epsilon),
\end{equation*}
where $J:\mathrm{T}\,G\to\mathrm{T}\,G$ is the CR structure.
\end{defn}
For every point $p\in\Sigma$, $\{(E_1)_{|p},(E_2)_{|p}\}$ is an orthonormal basis for the tangent plane $T_{p}\Sigma$.
For further reference, we will need the following.
\begin{lem}\label{l,r limits}
    Let $\overline{p}$, $\overline{q}$, $l_\epsilon$ and $\overline{r_\epsilon}$ as above.  Then, when $\epsilon\to  0^+$ we have the following:
     \begin{eqnarray*}
     &&
        l_\epsilon\to\|\nabla_H u\|,\quad
        \overline{r_\epsilon}\to 0,\quad
        \frac{\overline{r_\epsilon}}{l_\epsilon}\to 0,\quad
        \\
        &&
        \frac{\overline{r_\epsilon}}{\epsilon\, l_\epsilon}\to\frac{Tu}{\|\nabla_H u\|^2},
        \left(\frac{\overline{r_\epsilon}}{\epsilon}\right)^2\to\frac{(Tu)^2}{\|\nabla_H u\|^2},\quad
        \frac{\overline{r_\epsilon}}{\epsilon^2}\sim\frac{Tu}{\epsilon \|\nabla_H u\|}.
    \end{eqnarray*}
\end{lem}
\begin{lem}\label{MF ids}
We let:
$$
H^H=X(\overline{p})+Y(\overline{q}),\quad Q^H=X(\overline{q})-Y(\overline{p}).
$$
The following hold:
\begin{align*}
    n_\Sigma(\overline{p})&=(l/l_\epsilon)(-\overline{q}\,Q^H+\overline{r}\,T_\epsilon(\overline{p})),
    \quad &&n_\Sigma(\overline{q})=(l/l_\epsilon)(\overline{p}\,Q^H+\overline{r}\,T_\epsilon(\overline{q})),\\
    E_1(\overline{p})&=-\overline{q}\, H^H,\quad
    &&E_1(\overline{q})=\overline{p}\, H^H,\\
    E_2(\overline{p})&=(l/l_\epsilon)(-\overline{r}\,\overline{q}\,Q^H-T_\epsilon(\overline{p})),\quad
    &&E_2(\overline{q})=(l/l_\epsilon)(\overline{r}\,\overline{p}\,Q^H-T_\epsilon(\overline{q})).
\end{align*}    
\end{lem}
\begin{defn}\label{E_1, E_2, dual}
 The coframe dual to the frame of Definition  \ref{E_1, E_2} comprises the differential $1$-forms $\{\alpha_1,\alpha_2,\alpha_\Sigma\}$ given by
\begin{equation*}
    \begin{aligned}\label{coframe}
    \alpha_1&= -\overline{q}\,\omega_1+\overline{p}\,\omega_2,\\
    \alpha_2&=(l/l_\epsilon)(\overline{r}\,\overline{p}\,\omega_1+\overline{r}\,\overline{q}\,\omega_2-\vartheta_\epsilon),\\
    \alpha_\Sigma&=(l/l_\epsilon)\left(\overline{p}\,\omega_1+\overline{q}\,\omega_2+\overline{r}\,\vartheta_\epsilon\right).
\end{aligned}
\end{equation*}
\end{defn}
The following proposition holds:
\begin{prop}\label{wedge relations coframe}
    Let $\omega_1$, $\omega_2$, $\vartheta_\epsilon$, $\alpha_1$, $\alpha_2$ and $\alpha_\Sigma$ be as above. Then:
\begin{equation*}
        \begin{aligned}\label{eq: omega}
    \omega_1&=(l/l_\epsilon)\overline{p}\alpha_\Sigma-\overline{q}\alpha_1+(l/l_\epsilon)\overline{rp}\alpha_2,\\
    \omega_2&=(l/l_\epsilon)\overline{q}\alpha_\Sigma+\overline{p}\alpha_1+(l/l_\epsilon)\overline{rq}\alpha_2,\\
        \vartheta_\epsilon&=(l/l_\epsilon)(\overline{r}\alpha_\Sigma-\alpha_2).
    \end{aligned}
\end{equation*}
    Also,
\begin{equation*}
        \begin{aligned}\label{eq: wedges}
        \omega_1\wedge\omega_2&=(l/l_\epsilon)(\alpha_\Sigma\wedge\alpha_1-\overline{r}\alpha_1\wedge\alpha_2),\\
        \omega_1\wedge\vartheta_\epsilon&=(l/l_\epsilon)\overline{qr}\,\alpha_\Sigma\wedge\alpha_1-\overline{p}\,\alpha_\Sigma\wedge\alpha_2+(l/l_\epsilon)\,\overline{q}\,\alpha_1\wedge\alpha_2,\\
        \omega_2\wedge\vartheta_\epsilon&=-(l/l_\epsilon)\overline{pr}\,\alpha_\Sigma\wedge\alpha_1-\overline{q}\,\alpha_\Sigma\wedge\alpha_2-(l/l_\epsilon)\,\overline{p}\,\alpha_1\wedge\alpha_2
    \end{aligned}
\end{equation*}
    and
    \begin{equation*}
    \begin{aligned}\label{eq: d new}
     d\omega_1=&-(l/l_\epsilon)(a_3+\epsilon\overline{r}(\overline{q}a_1-\overline{p}a_2))\,\alpha_\Sigma\wedge\alpha_1\\
     &+\epsilon((\overline{p}a_1+\overline{q}a_2))\,\alpha_\Sigma\wedge\alpha_2\\
     &-(l/l_\epsilon)(-\overline{r}a_3+\epsilon\overline{r}(\overline{q}a_1-\overline{p}a_2))\,\alpha_1\wedge\alpha_2,\\
     d\omega_2=&-(l/l_\epsilon)(b_3+\epsilon\overline{r}(\overline{q}b_1-\overline{p}b_2))\,\alpha_\Sigma\wedge\alpha_1\\
     &+\epsilon((\overline{p}b_1+\overline{q}b_2))\,\alpha_\Sigma\wedge\alpha_2\\
     &-(l/l_\epsilon)(-\overline{r}b_3+\epsilon\overline{r}(\overline{q}b_1-\overline{p}b_2))\,\alpha_1\wedge\alpha_2,\\
     d\vartheta_\epsilon=&-(cl/(\epsilon l_\epsilon))(\alpha_\Sigma\wedge\alpha_1-\overline{r}\alpha_1\wedge\alpha_2).
    \end{aligned}
    \end{equation*}
\end{prop}
\begin{prop}\label{McoF ext der}
    Let $\alpha_1$, $\alpha_2$ and $\alpha_\Sigma$ be as above; then the differentials $d\alpha_1$, $d\alpha_2$ and $d\alpha_\Sigma$ are given by:    
    \begin{eqnarray*}
    d\alpha_1&=&A_1\,\alpha_\Sigma\wedge\alpha_1+A_2\,\alpha_\Sigma\wedge\alpha_2+A_3\,\alpha_1\wedge\alpha_2,\\
    d\alpha_2&=&B_1\,\alpha_\Sigma\wedge\alpha_1+B_2\,\alpha_\Sigma\wedge\alpha_2+B_3\,\alpha_1\wedge\alpha_2,\\
    d\alpha_\Sigma&=&C_1\,\alpha_\Sigma\wedge\alpha_1+C_2\,\alpha_\Sigma\wedge\alpha_2,
\end{eqnarray*}
where:
\begin{equation}
    \begin{aligned}\label{ABC}
    A_1&=(l/l_\epsilon)\left(H^H+a_3\overline{q}-b_3\overline{p}+\epsilon \overline{r}(b_2\overline{p}^2+a_1\overline{q}^2-(a_2+b_1)\overline{pq})\right),\\
    A_2&=\overline{q}T_\epsilon(\overline{p})-\overline{p}T_\epsilon(\overline{q})+\epsilon\left((b_2-a_1)\overline{pq}+b_1\overline{p}^2-a_2\overline{q}^2\right),\\
    A_3&=(l/l_\epsilon)\left(-\overline{r}(H^H+a_3\overline{q}-b_3\overline{p})+\epsilon \overline{r}(b_2\overline{p}^2+a_1\overline{q}^2-(b_1+a_2)\overline{pq})\right),\\
    B_1&=\overline{r}E_1(\log(l_\epsilon/l))-E_1(\overline{r})+c/\epsilon,\\
    B_2&=n_\Sigma(\log(l_\epsilon/l))-\overline{r}E_2(\log(l_\epsilon/l))-E_2(\overline{r}),\\
    B_3&=E_1(\log(l_\epsilon/l))-\overline{r}(c/\epsilon),\\
    C_1&=E_1(\log l_\epsilon),\\
    C_2&=E_2(\log l_\epsilon).
\end{aligned}
\end{equation}
\end{prop}
\begin{proof}
Since $\alpha_\Sigma=du/l_\epsilon$, we have
\begin{eqnarray*}
    d\alpha_\Sigma=-(1/l_\epsilon^2) dl_\epsilon\wedge  du&=&
    \alpha_\Sigma\wedge d(\log l_\epsilon)\\
    &=&E_1(\log l_\epsilon)\,\alpha_\Sigma\wedge \alpha_1+
    E_2(\log l_\epsilon)\,\alpha_\Sigma\wedge \alpha_2.
\end{eqnarray*}
Next,
\begin{eqnarray*}
    d\alpha_1&=&-d\overline{q}\wedge\omega_1+d\overline{p}\wedge\omega_2-\overline{q}\,d\omega_1+\overline{p}\,d\omega_2\\
    &=&-\left(n_\Sigma(\overline{q})\alpha_\Sigma+E_1(\overline{q})\alpha_1+E_2(\overline{q})\alpha_2\right)\wedge\left((l/l_\epsilon)\overline{p}\alpha_\Sigma-\overline{q}\alpha_1+(l/l_\epsilon)\overline{rp}\alpha_2\right)\\
    &&+\left(n_\Sigma(\overline{p})\alpha_\Sigma+E_1(\overline{p})\alpha_1+E_2(\overline{p})\alpha_2\right)\wedge\left((l/l_\epsilon)\overline{q}\alpha_\Sigma+\overline{p}\alpha_1+(l/l_\epsilon)\overline{rq}\alpha_2\right)\\
    &&+\overline{q}(l/l_\epsilon)(a_3+\epsilon\overline{r}(\overline{q}a_1-\overline{p}a_2))\,\alpha_\Sigma\wedge\alpha_1\\
     &&-\overline{q}\epsilon((\overline{p}a_1+\overline{q}a_2))\,\alpha_\Sigma\wedge\alpha_2\\
     &&+\overline{q}(l/l_\epsilon)(-\overline{r}a_3+\epsilon\overline{r}(\overline{q}a_1-\overline{p}a_2))\,\alpha_1\wedge\alpha_2,\\
     &&-\overline{p}(l/l_\epsilon)(b_3+\epsilon\overline{r}(\overline{q}b_1-\overline{p}b_2))\,\alpha_\Sigma\wedge\alpha_1\\
     &&+\overline{p}\epsilon((\overline{p}b_1+\overline{q}b_2))\,\alpha_\Sigma\wedge\alpha_2\\
     &&-\overline{p}(l/l_\epsilon)(-\overline{r}b_3+\epsilon\overline{r}(\overline{q}b_1-\overline{p}b_2))\,\alpha_1\wedge\alpha_2 \\
     &=& A_1\,\alpha_\Sigma\wedge\alpha_1+A_2\,\alpha_\Sigma\wedge\alpha_2+A_3\,\alpha_1\wedge\alpha_2,
\end{eqnarray*}
where $A_i$ for $i\in \{1,2,3\}$ are defined as follows: using the identities of Lemma \ref{MF ids} the coefficient of $\alpha_\Sigma\wedge\alpha_1$ is
\begin{eqnarray*}
    A_1&=&\overline{q}n_\Sigma(\overline{q})+\overline{p}n_\Sigma(\overline{p})\\
    &&+(l/l_\epsilon)(\overline{p}E_1(\overline{q})-\overline{q}E_1(\overline{p}))\\
    &&+(l/l_\epsilon)(a_3\overline{q}-b_3\overline{p}+\epsilon\overline{r}(a_1\overline{q}^2+b_2\overline{p}^2-\overline{qp}(a_2+b_1)))\\
&=&(l/l_\epsilon)\left(H^H+a_3\overline{q}-b_3\overline{p}+\epsilon \overline{r}(b_2\overline{p}^2+a_1\overline{q}^2-(a_2+b_1)\overline{pq})\right).
\end{eqnarray*}
Similarly, the coefficient of $\alpha_\Sigma\wedge\alpha_2$ is
\begin{eqnarray*}
    A_2&=&(l/l_\epsilon)\left(-\overline{rp}n_\Sigma(\overline{q})+\overline{rq}n_\Sigma(\overline{p})+\overline{p}E_2(\overline{q})-\overline{q}E_2(\overline{p})\right)\\
    &&+\epsilon\left((b_2-a_1)\overline{pq}+b_1\overline{p}^2-a_2\overline{q}^2\right)\\
    &=&\overline{q}T_\epsilon(\overline{p})-\overline{p}T_\epsilon(\overline{q})+\epsilon\left((b_2-a_1)\overline{pq}+b_1\overline{p}^2-a_2\overline{q}^2\right).
\end{eqnarray*}
Finally, the coefficient of $\alpha_1\wedge\alpha_2$ is
\begin{eqnarray*}
    A_3&=&(l/l_\epsilon)\left(\overline{r}(\overline{q}E_1(\overline{p})-\overline{p}E_1(\overline{q})+b_3p-a_3q)\right)-\overline{q}E_2(\overline{q})-\overline{p}E_2(\overline{p})\\
    &&+(l/l_\epsilon)\epsilon \overline{r}\left(b_2\overline{p}^2+a_1\overline{q}^2-(b_1+a_2)\overline{pq})\right),\\
    &=&(l/l_\epsilon)\left(-\overline{r}(H^H+a_3\overline{q}-b_3\overline{p})+\epsilon \overline{r}(b_2\overline{p}^2+a_1\overline{q}^2-(b_1+a_2)\overline{pq})\right),
\end{eqnarray*}
and the formula for $d\alpha_1$ follows.
Finally, to calculate $d\alpha_2$, we first observe that
$$
\alpha_2=\overline{r}\alpha_\Sigma-\frac{l_\epsilon}{l}\vartheta_\epsilon,
$$
hence
\begin{eqnarray*}
d\alpha_2&=&d\overline{r}\wedge\alpha_\Sigma+\overline{r}d\alpha_\Sigma-d(l_\epsilon/l)\wedge \vartheta_\epsilon-(l_\epsilon/l) d\vartheta_\epsilon\\
&=&\left(n_\Sigma(\overline{r})\,\alpha_\Sigma+E_1(\overline{r})\,\alpha_1+E_2(\overline{r})\,\alpha_2\right)\wedge\alpha_\Sigma\\
&&+\overline{r}E_1(\log l_\epsilon)\,\alpha_\Sigma\wedge \alpha_1+\overline{r}E_2(\log l_\epsilon)\,\alpha_\Sigma\wedge \alpha_2\\
&&-\left( n_{\Sigma} (l_\epsilon/l)\,\alpha_{\Sigma}+E_1(l_\epsilon/l)\,\alpha_1 +E_2 (l_\epsilon/l)\, a_2 \right) \wedge \left( (l/l_\epsilon)\overline{r} \alpha_{\Sigma} -(l/l_\epsilon) \,\alpha_2 \right) \\
&&+(c/\epsilon )(\alpha_\Sigma\wedge\alpha_1-\overline{r}\alpha_1\wedge\alpha_2)\\
&=&-E_1(\overline{r})\,\alpha_\Sigma\wedge\alpha_1-
    E_2(\overline{r})\,\alpha_\Sigma\wedge\alpha_2\\
    &&+\overline{r}E_1(\log l_\epsilon)\,\alpha_\Sigma\wedge \alpha_1+\overline{r}E_2(\log l_\epsilon)\,\alpha_\Sigma\wedge \alpha_2\\
    &&+\overline{r}E_1(\log(l_\epsilon/l))\,\alpha_\Sigma\wedge\alpha_1+\left(n_\Sigma(\log(l_\epsilon/l))+\overline{r}E_2(\log(l_\epsilon/l))\right)\,\alpha_\Sigma\wedge\alpha_2 +E_1(\log(l_\epsilon/l))\,\alpha_1\wedge\alpha_2\\
    &&+(c/\epsilon )(\alpha_\Sigma\wedge\alpha_1-\overline{r}\alpha_1\wedge\alpha_2) \\
    &=& B_1\,\alpha_\Sigma\wedge\alpha_1+B_2\,\alpha_\Sigma\wedge\alpha_2+B_3\,\alpha_1\wedge\alpha_2.
\end{eqnarray*}
    \end{proof}
\begin{cor}
The Lie brackets   
$[n_\Sigma,E_i]$, $i\in\{1,2\},$\ and $[E_1,E_2]$,
are given by
\begin{eqnarray*}
    [n_\Sigma,E_1]&=&-C_1n_\Sigma-A_1E_1-B_1E_2,
    \end{eqnarray*}
    \begin{eqnarray*}
[n_\Sigma,E_2]&=&-C_2n_\Sigma-A_2E_1-B_2E_2,
    \end{eqnarray*}
    \begin{eqnarray*}
[E_1,E_2]&=&-A_3E_1-B_3E_2.\\
\end{eqnarray*}
\end{cor}
\begin{proof}
We will only calculate $[E_1,E_2]$, since the other two identities are proved in an analogous manner. We have
$$
[E_1,E_2]=\alpha_\Sigma([E_1,E_2])\,n_\Sigma+\alpha_1([E_1,E_2])\,E_1+\alpha_2([E_1,E_2])\,E_2.
$$
On the other hand, formula \eqref{domegaXY} yields
\begin{eqnarray*}
    d\alpha_\Sigma(E_1,E_2)&=&-\alpha_\Sigma([E_1,E_2]),\\
    d\alpha_1(E_1,E_2)&=&-\alpha_1([E_1,E_2]),\\
    d\alpha_2(E_1,E_2)&=&-\alpha_2([E_1,E_2]),
\end{eqnarray*}
and our formula follows.
\end{proof}
Denote by $\nabla$ the Levi-Civita connection of $(G,g_\epsilon)$ and  consider the orthonormal frame $\{E_1,E_2,n_\Sigma\}$. Let $j\in\{1,2\}$, $k\in\{1,2,3\}$ and let $X_k\in\{E_1,E_2,n_\Sigma\}$. Recall that the {\it connection $1$-forms} $\eta_j^k$ are the differential $1$-forms defined by:
\begin{equation*}
    \eta_j^k(E_i)=-g_\epsilon(\nabla_{E_i}X_k,E_j),\quad i\in{1,2}.
\end{equation*}
By the definition of the Levi-Civita connection it follows that $\eta^k_j=-\eta_j^k$, for all $k$ and $j$.
\begin{prop}\label{Cartan I SE}
Let $\alpha_1$, $\alpha_2$ and $\alpha_\Sigma$ be defined as above, let $j\in\{1,2\}$, $k\in\{1,2,3\}$ and let $\eta_j^k$ be the connection $1$-forms. Then the following hold:
    \begin{eqnarray*}
    d\alpha_1 &=& \eta^2_1 \wedge \alpha_2 +\eta^3_1 \wedge \alpha_{\Sigma}, \\
    d\alpha_2 &=& -\eta^2_1 \wedge \alpha_1 +\eta^3_2 \wedge\alpha_{\Sigma}, \\
    d\alpha_{\Sigma} &=& -\eta^3_1 \wedge \alpha_1 -\eta^3_2 \wedge\alpha_2,
\end{eqnarray*}
where
\begin{equation}
    \begin{aligned}\label{conn 1-forms}
    \eta_1^2&=A_3\,\alpha_1+B_3\,\alpha_2+\frac{A_2-B_1}{2}\,\alpha_\Sigma,\\
    \eta_1^3&=-A_1\,\alpha_1+\frac{-A_2-B_1}{2}\,\alpha_2-C_1\,\alpha_\Sigma,\\
    \eta_2^3&=\frac{-A_2-B_1}{2}\,\alpha_1-B_2\,\alpha_2-C_2\,\alpha_\Sigma,
\end{aligned}
\end{equation}
and $A_1,A_2,A_3$, $B_1,B_2,B_3$ and $C_1,C_2$ are given by equations \eqref{ABC}.
\end{prop}
\begin{proof}
  We write the $1$-forms $\eta_j^k$ as linear combinations of the elements of the coframe $\{\alpha_1,\alpha_2,\alpha_\Sigma\}$:
\begin{eqnarray*}
\eta^2_1 &=& (\eta^2_1)_1\ \alpha_1+(\eta^2_1)_2\ \alpha_2+(\eta^2_1)_3\ \alpha_{\Sigma} \\
\eta^3_1 &=& (\eta^3_1)_1\ \alpha_1+(\eta^3_1)_2\ \alpha_2+(\eta^3_1)_3\ \alpha_{\Sigma} \\
\eta^3_2 &=& (\eta^3_2)_1\ \alpha_1+(\eta^3_2)_2\ \alpha_2+(\eta^3_2)_3\ \alpha_{\Sigma}.
\end{eqnarray*}  
Then, combining Proposition \ref{McoF ext der} and Proposition \ref{Cartan I SE} we obtain 
\begin{eqnarray*}
    &&
    (\eta_1^2)_1=A_3,\quad (\eta_1^2)_3-(\eta_1^3)_2=A_2,\quad (\eta_1^3)_1=-A_1,\\
    &&
    (\eta_1^2)_3+(\eta_2^3)_1=-B_1,\quad (\eta_2^3)_2=-B_2,\quad (\eta_1^2)_2=B_3,\\
    &&
    (\eta_1^3)_3=-C_1,\quad( \eta_2^3)_3=-C_2,\quad (\eta_1^3)_2-(\eta_2^3)_1=0,
\end{eqnarray*}
The latter relations result in \eqref{conn 1-forms}.
\end{proof}
The {\it sectional curvatures} of characteristic planes spanned by $E_1,E_2$, $E_1,n_\Sigma$ and $E_2,n_\Sigma$ are given respectively by: 
    \begin{eqnarray*}
     \overline{K}^\epsilon(E_1,E_2) &=& \eta^3_2\wedge\eta^3_1(E_1,E_2)-d\eta^2_1(E_1,E_2),\\
     \overline{K}^\epsilon(E_1,n_\Sigma) &=&-d\eta^3_1(E_1,n_\Sigma)-\eta^3_2\wedge\eta^2_1(E_1,n_\Sigma) ,\\
     \overline{K}^\epsilon(E_2,n_\Sigma) &=& \eta^3_1\wedge\eta^2_1(E_2,n_\Sigma)-d\eta^3_2(E_2,n_\Sigma).
    \end{eqnarray*}
We apply the above formulae to the forms given by \eqref{conn 1-forms} to obtain the following:
\begin{prop} 
Let $\Sigma$ be a regular surface in $(G,g_\epsilon)$ defined by $u=0$ where $u:G\to\R$ is a $\mathcal{C}^2$ function. Let $\{E_1,E_2,n_\Sigma\}$ be the associated frame and let  $A_1$, $A_2$, $A_3$, $B_1$, $B_2$, $B_3$, $C_1$ and $C_2$ be as in \eqref{ABC}. 
Then:
    \begin{eqnarray*}
     \overline{K}^\epsilon(E_1,E_2) &=&  -E_1(B_3)+E_2(A_3)-A_3^2-B_3^2+\frac{(A_2+B_1)^2}{4}-A_1B_2 ,\\
     \overline{K}^\epsilon(E_1,n_\Sigma) &=& E_1(C_1)-n_\Sigma(A_1)-A_1^2-C_1^2+\frac{(A_2+B_1)^2}{4}+C_2A_3,\\
     \overline{K}^\epsilon(E_2,n_\Sigma) &=&
     E_2(C_2)-n_\Sigma(B_2)-B_2^2-C_2^2+\frac{(A_2+B_1)^2}{4}+C_1B_3.
    \end{eqnarray*}
\end{prop}
The second fundamental form $\mathbf{II}^\epsilon$ of $\Sigma$ is given by:
\begin{equation*}
\mathbf{II}^\epsilon=\begin{pmatrix}
   -g_\epsilon(\nabla^\epsilon_{E_1}n_\Sigma, E_1) & -g_\epsilon(\nabla^\epsilon_{E_1}n_\Sigma, E_2) \\
   \\
   -g_\epsilon(\nabla^\epsilon_{E_2}n_\Sigma, E_1) & -g_\epsilon(\nabla^\epsilon_{E_2}n_\Sigma, E_2)
\end{pmatrix}    .
\end{equation*}
Using \eqref{conn 1-forms}, we calculate explicitly:
\begin{equation*}
\mathbf{II}^\epsilon=\begin{pmatrix}
    \eta_1^3(E_1) & \eta_2^3(E_1) \\
    \\
    \eta_2^3(E_1) & \eta_2^3(E_2)
\end{pmatrix}=\begin{pmatrix}
   -A_1 & -\frac{1}{2}(A_2+B_1) \\
   \\
   -\frac{1}{2}(A_2+B_1) & -B_2
\end{pmatrix}    .
\end{equation*}
The \textit{Riemannian mean curvature} $H^\epsilon$ of $\Sigma$ is
\begin{equation*}
   H^\epsilon_\Sigma:=-\text{Trace}(\mathbf{II^\epsilon})=A_1+B_2.
\end{equation*}
Furthermore, the \textit{Gaussian curvature} $K^\epsilon_\Sigma$ is given by the Gauss equation (see Theorem 2.5, Chapter 6 in \cite{DC}):
\begin{equation*}
K^\epsilon_\Sigma:={\overline K}^\epsilon(E_1,E_2)+\det({\bf II}^\epsilon)=-E_1(B_3)+E_2(A_3)-A_3^2-B^2_3.   
\end{equation*}
\subsection{Notions of curature}\label{sec:not-curv}
The preparatory discussion of the previous section now leads to the following main definitions. 
\begin{defn}\label{Def-Curvatures}
    Let $\Sigma$ be a regular surface in $G$.
    Away from the characteristic set $\mathcal{C}(\Sigma)$:
    \begin{enumerate}
        \item the \textit{horizontal mean curvature} $H^h_\Sigma$ of $\Sigma$ is given by
    \begin{equation}\label{HMC}
        H^h_\Sigma:=\lim_{\epsilon\to 0^+}H^\epsilon_\Sigma=X\left( \frac{Xu}{\|\nabla_H u\|}\right)+Y\left( \frac{Yu}{\|\nabla_H u\|} \right) +\frac{a_3Yu-b_3Xu}{\|\nabla_H u\|};
    \end{equation}
    \item the \textit{horizontal Gaussian curvature} $K^h_\Sigma$ of $\Sigma$ is given by
    \begin{equation}\label{HGC}
    K_\Sigma^h:=\lim_{\epsilon\to0^+}K^\epsilon_\Sigma=E_1\left(c\frac{Tu}{\|\nabla_Hu\|}\right)-\left(c\frac{Tu}{\|\nabla_Hu\|}\right)^2.
    \end{equation}
    \end{enumerate}
\end{defn}
\subsubsection{Invariance}\label{sec:curv-invariance} The horizontal mean curvature $H^h_\Sigma$  of a hypersurface $\Sigma$, with defining function $u=0$, depends only on the  horizontal derivatives  of $u$, that is, it remains invariant under the action of ${\rm Isom}_{CC}(G)$. We show in the next proposition that the same holds for the horizontal Gauss curvature; we state a definition first:
 \begin{defn}\label{QS}
    Let $\Sigma$ be a regular surface in $G$ and let us consider $Q^H=X(\overline{q})-Y(\overline{p})$, the {\it symplectic distortion} of $\Sigma$ is defined as   
    $$
    Q_\Sigma^h=Q^H-a_3\overline{p}-b_3\overline{q}.
$$ 
\end{defn}
It is evident from the definition that $Q_\Sigma^h$ is invariant under the action of ${\rm Isom}_{CC}(G)$.
The following holds.
\begin{prop}{\bf (Horizontal Theorema Egregium.)} 
    The horizontal Gauss curvature $K_\Sigma^h$ of a hypersurface $\Sigma$ embedded in the  group $G$ defined by $u=0$ where $u$ is a $\mathcal{C}^2$ function, depends only on the horizontal derivatives of $u$. In particular,
    \begin{equation}\label{K-Q}
   K_\Sigma^h=E_1(Q_\Sigma^h)-E_1E_1(\log \|\nabla_Hu\|)-\left(Q_\Sigma^h-E_1(\log \|\nabla_Hu\|\right)^2,        \end{equation}
   where $Q_\Sigma^h$ is the symplectic distortion of
   $\Sigma$.
\end{prop}
\begin{proof}
We have
\begin{eqnarray*}
 c\frac{Tu}{\|\nabla_Hu\|}&=&\frac{[X,Y]u-a_3Xu-b_3Yu}{\|\nabla_Hu\|}  \\
 &=&\frac{[X,Y]u}{\|\nabla_Hu\|}-a_3\overline{p}-b_3\overline{q}.
\end{eqnarray*}
Let $l=\|\nabla_Hu\|$. Since
\begin{eqnarray*}
    [X,Y]u&=&X(Yu)-Y(Xu)\\
    &=&X(l\overline{q})-Y(l\overline{p})\\
    &=&-E_1(l)+l(X(\overline{q})-Y(\overline{p}))\\
    &=&-E_1(l)+lQ^H,
\end{eqnarray*}
we conclude that
\begin{equation}\label{Q-aux}
   c\frac{Tu}{\|\nabla_Hu\|}=Q^H-a_3\overline{p}-b_3\overline{q}-E_1(\log \|\nabla_Hu\|).
   \end{equation}
Formula (\ref{K-Q}) follows.
    \end{proof}
    \begin{cor}\label{cor-cr-K}
    The horizontal Gauss curvature $K_\Sigma^h$ of a surface $\Sigma$ of $G$ remains invariant under the action of ${\rm Isom}_{CC}(G)$.
    \end{cor}
\begin{prop}\label{prop-K-Q-ac}
    If $\Sigma$ is a regular surface in $G$ that satisfies $Tu=1$, then the following holds:
    \begin{equation}
        K_\Sigma^h=-\frac{c}{\|\nabla_Hu\|}Q_\Sigma^h.
    \end{equation}
\end{prop}
\begin{proof}
Let $\Sigma$ be a regular surface in $G$ so that $Tu\equiv 1$.
Then formula \eqref{HGC} reads
$$
K_\Sigma^h=-\frac{c}{\|\nabla_Hu\|^2}\left(E_1(\|\nabla_Hu\|)+c\right),
$$
and on the other hand,
Eq. \eqref{Q-aux} gives 
$$
Q_\Sigma^h=\frac{1}{\|\nabla_Hu\|}\left(E_1(\|\nabla_Hu\|)+c\right).
$$
The result follows.
\end{proof}
\begin{cor}\label{ZeroQ-K-Cor}
    A regular surface $\Sigma$ in $G$ satisfying $Tu=1$ has zero horizontal Gaussian curvature if and only if one the following hold: a) it has zero symplectic distortion, and b) $E_1(\|\nabla_Hu\|)=-c$.
\end{cor}
\begin{rem}
    Note that a surface satisfying $Tu\neq 0$ can be normalised so that $Tu=1.$
\end{rem}
\subsubsection{Horizontal geodesic curvature and horizontal normal curvature}\label{sec:hor-geod}
\begin{defn}
    The {\it horizontal geodesic curvature} $k_g^h$ and the {\it horizontal normal curvature} $k_n^h$ of a curve $\gamma$ on $\Sigma$ are defined respectively by
    \begin{eqnarray*}
        k_g^h &=&\lim_{\epsilon \to 0^+} \eta^2_1 (\dot{\gamma}),\\
        k_n^h &=& -\lim_{\epsilon \to 0^+} \eta^3_1 (\dot{\gamma}).
    \end{eqnarray*}
\end{defn}
\begin{prop}\label{k-H}
Let $\gamma$ be an integral curve of $E_1$ on a regular surface $\Sigma$ in $G$. Then at non-characteristic points we have 
 $$
        k_g^h =0 \,\,\text{and}\,\,
        k_n^h  = H_\Sigma^h,
$$
   along $\gamma$.
\end{prop}
\begin{proof}
Let $\gamma$ be an integral curve of $E_1$ regular surface $\Sigma$ in $G$ away from characteristic points, then
    \begin{eqnarray*}
    \eta^2_1(\dot\gamma)&=&A_3\to 0,\,\text{as}\,\epsilon\to 0^+,\\
    -\eta^3_1(\dot\gamma)&=&A_1\to H^H+a_3\overline{q}-b_3\overline{p}=H_\Sigma^h,\,\,\text{as}\,\,\epsilon\to 0^+.
\end{eqnarray*}
\end{proof}

\subsubsection{$G$-cylindrical surfaces}\label{sec:G-cyl}
\begin{defn}
   Let $\Sigma$ be a regular surface given by $u=0$. It shall be called {\it $G$-cylindrical} if $Tu\equiv 0$.
\end{defn}
It is clear that when $Tu=0$ then $K_\Sigma^h=0$. 
\section{Applications: two characteristic examples}\label{sec-appl}
In this section, we apply our previous results to two characteristic examples of contact Lie groups with a sub-Riemannian structure. The first one is the Heisenberg group, whose general features we describe in Section \ref{sec:ex-Heis}. In Section \ref{sec-Hsurfaces}, we consider characteristic examples of surfaces embedded in the Heisenberg group. In particular, for surfaces of revolution (see Section \ref{sec-Hsurfrev}) we give a complete classification surfaces with constant horizontal Gauss curvature (Theorem \ref{thm:H-constK}, compare to Section 13 in \cite{Vel-RSCC1}), constant horizontal mean curvature (Theorem \ref{thm:H-constH}, compare to Theorem 5.4 in \cite{RR-Rot.inv.}) and constant symplectic distortion (Theorem \ref{thm:H-constQ}). Next, is the example of the affine-additive group $\Aa$, see Section \ref{Sec:ex-AA}. Examples of surfaces embedded in $\Aa$ may be found in Section \ref{Sec:AA-surf}. In Section \ref{sec-AAsurfrev} we define surfaces of revolution in $\Aa$ and in a manner analogous with the one used in the example of the Heisenberg group we classify completely surfaces of revolution with a) constant horizontal Gauss curvature (Theorem \ref{thm:AA-constK}), b) constant horizontal mean curvature (Theorem \ref{thm:AA-constH}) and c) constant symplectic distortion (Theorem \ref{thm:AA-constQ}).
\subsection{Heisenberg group}\label{sec:ex-Heis}
The \textit{Heisenberg group} $\mathbb{H}$ is the Lie group with underlying space $\R^3$ with coordinates $(x,y,t)$ and group law which is defined as follows: for $p=(x,y,t)$ and $p'=(x',y',t')$,
$$
p\cdot p'=(x+x',\,y+y',\,t+t'+2(yx'-xy')).
$$
An orthonormal basis for the tangent bundle $\rm{T}\,\mathbb H$ comprises the left-invariant vector fields
$$
X=\partial_x+2y\partial_t,\quad Y=\partial_y-2x\partial_t,\quad T=\partial_t.
$$
The only non-zero Lie bracket relation is
$
[X,Y]=-4T
$
and, following the notation from Proposition \ref{Prop: abc}, we have
$
a_i=b_i=0,$ $i\in\{1,2,3\},$ and $c=-4.$
The contact form is 
$$
\theta=dt+2x\,dy-2y\,dx,
$$
so that $X,Y\in\ker\vartheta$ and $T$ is the Reeb field. The corresponding coframe is
$$
\omega_1=dx,\quad \omega_2=dy,\quad \vartheta.
$$
The sectional curvatures of characteristic planes are calculated by using \eqref{SecCurvXY},\eqref{SecCurvXT} and \eqref{SecCurvYT}: 
\begin{eqnarray*}
    K_\epsilon(X,Y)=-\frac{12}{\epsilon^2},\quad K_\epsilon(X,T_\epsilon)=K_\epsilon(Y,T_\epsilon)=\frac{4}{\epsilon^2}.
\end{eqnarray*}
Moreover, according to Theorem \ref{Thm:Isom} we obtain
\begin{equation*}\label{eq: Isom Heis}
    {\rm Isom}_{CC}(\mathbb H) = \mathbb H \rtimes {\rm AutIsom}_{CC}(\mathbb H)\,.
\end{equation*}
In particular, the group ${\rm Isom}_{CC}(\mathbb{H})$ is generated by the following transformations:
\begin{enumerate}
    \item left translations $L_{(x_0,y_0,t_0)}:\mathbb{H}\to\mathbb{H}$ with respect to $(x_0,y_0,t_0)\in\mathbb H$, given by $$ L_{(x_0,y_0,t_0)}(x,y,t)= (x_0,y_0,t_0)\cdot(x,y,t);$$
    \item rotations $R_\theta:\mathbb{H}\to\mathbb{H}$ in an angle $\theta$ around the $t$-axis, given by $$R_\theta(x,y,t)= (x\cos\theta-y\sin\theta,\,x\sin\theta+y\cos\theta,\,t);$$
\item conjugation $j:\mathbb{H}\to\mathbb{H}$, given by
$$j(x,y,t)= (x,-y, -t).$$
\end{enumerate}
We also mention dilations $D_\delta:\mathbb{H}\to\mathbb{H}$ of factor $\delta>0$, given by
$$
D_\delta(x,y,t)=(\delta x,\delta y,\delta^2t).
$$
A dilation $D_\delta$ is an automorphism of $\mathbb{H}$ that scales the $CC$-distance up to the factor $\delta$, i.e. for all $p,q\in\mathbb{H}$,
$$
d_{CC}(D_\delta(p),D_\delta(q))=\delta\,d_{CC}(p,q).
$$
\subsubsection{Surfaces embedded in the Heisenberg group}\label{sec-Hsurfaces}
Let $\Sigma$ be a regular hypersurface defined by the equation $u=0$. Then the horizontal Gauss curvature $K_\Sigma^h$ of  $\Sigma$ is given by
    \begin{equation}\label{KH-Heis}
    K_\Sigma^h=-4 E_1\left(\frac{Tu}{\|\nabla_Hu\|}\right)-16\left(\frac{Tu}{\|\nabla_Hu\|}\right)^2,
    \end{equation} 
    the horizontal mean curvature $H_\Sigma^h$ of  $\Sigma $ is given by
    \begin{equation}\label{HH-Heis}
     H^h_\Sigma=X\left(\frac{Xu}{\|\nabla_Hu\|}\right)+ Y\left(\frac{Yu}{\|\nabla_Hu\|}\right),
    \end{equation}
    and the symplectic distortion $Q_\Sigma^h$ of  $\Sigma $ is given by
    \begin{equation}\label{SD-Heis}
     Q^h_\Sigma=X\left(\frac{Yu}{\|\nabla_Hu\|}\right)- Y\left(\frac{Xu}{\|\nabla_Hu\|}\right). 
    \end{equation}
    We have seen that horizontal Gauss curvature, horizontal mean curvature and symplectic distortion are all invariant under the action of isometries. In the case of the Heisenberg group, we have the following proposition concerning the action of dilations on an embedded surface.
\begin{prop}\label{K-HHinvariance}
Let $\Sigma$ be a regular surface in $\mH$ defined by the equation $\Sigma=\{p\in\mH:u(p)=0\}$ where $u:\mH\to \R$ is a $C^2$-smooth function. Let $D_\delta:\mH\to\mH$ be the dilation given by $$D_\delta(x,y,t)=(\delta x,\delta y,\delta ^2 t),$$ for each $(x,y,t)\in\mH$. Suppose that the surface $\Sigma'$ of $\mH$ is defined by the equation $\{p'\in\mH:u'(p')=0\}$, where 
$u'(p')=u(D_{1/\delta}(p'))=u(p)$. Then for each $p\in\Sigma$,
$$
H_{\Sigma'}(D_\delta(p))=\frac{1}{\delta}H_\Sigma^h(p),\quad K_{\Sigma'}(D_\delta(p))=\frac{1}{\delta}K_\Sigma^h(p),\quad Q_{\Sigma'}(D_\delta(p))=\frac{1}{\delta}Q_\Sigma^h(p).
$$
   \end{prop} 
\begin{proof}
We will prove only the first formula; the other two are derived in an analogous manner.
 Let $D_\delta$, $\delta>0$, and let the dilation given by
       $
       D_\delta(x,y,t)=(\delta x,\,\delta y,\,\delta^2t).
       $
       First, we have 
       $$
       (D_\delta)_{*}X=\delta X,\quad (D_\delta)_{*}Y=\delta Y,\quad (D_\delta)_{*}T=\delta^2 T,
       $$
       and therefore, for $p'=D_\delta(p)$ we have
       \begin{eqnarray*}
           &&
           Xu'(p')=\frac{1}{\delta} Xu(p),\quad Yu'(p')=\frac{1}{\delta} Yu(p), Tu'(p')=\frac{1}{\delta^2} Tu(p),\\
           &&
           \|\nabla_Hu'(p')\|=\frac{1}{\delta}\|\nabla_Hu(p)\|,\\
           &&
           XXu'(p')=\frac{1}{\delta^2}XXu(p),\quad
           XYu'(p')=\frac{1}{\delta^2}XYu(p),\\
           &&
           YXu'(p')=\frac{1}{\delta^2}YXu(p),\quad
           YYu'(p')=\frac{1}{\delta^2}YYu(p),\\
           &&
           XTu'(p')=\frac{1}{\delta^3}XTu(p),\quad
           YTu'(p')=\frac{1}{\delta^3}YTu(p).
       \end{eqnarray*}
       By expanding formula (\ref{HH-Heis}) we obtain after straightforward calculations that
       \begin{eqnarray*}
           H_{\Sigma'}(p')&=&\frac{XXu'(p')Yu'(p')^2+YYu'(p')Xu'(p')^2-Xu'(p')Yu'(p')(XYu'(p')+YXu'(p'))}{\|\nabla_Hu'(p')\|^3}\\
           &=&\frac{1}{\delta}\frac{XXu(p)Yu(p)^2+YYu(p)Xu(p)^2-Xu(p)Yu(p)(XYu(p)+YXu(p))}{\|\nabla_Hu(p)\|^3}\\
           &=&\frac{1}{\delta}H_\Sigma(p).
       \end{eqnarray*} 
   \end{proof}
 \begin{ex}
     {\bf $\mathbb{H}$-Cylindrical surfaces.}
     An $\mathbb{H}$-cylindrical surface $\Sigma$ is given by an equation of the form $u(x,y,t)=g(x,y)=0$, since $Tu=\partial_tu=0$. We also have that $\Sigma$ is also invariant under the action of vertical translations. The horizontal Gauss curvature of $\mathbb{H}$-cylindrical surfaces is zero and moreover,
$$
Xu=g_x,\quad Yu=g_y,\quad \|\nabla_Hu\|=\sqrt{g_x^2+g_y^2}.
$$
Hence
$$
H_\Sigma^h=\frac{g_{xx}g_y^2+g_{yy}g_x^2-2g_{xy}g_xg_y}{(g_x^2+g_y^2)^{3/2}},
$$
and
$$
Q_\Sigma^h=\frac{g_xg_y(g_{yy}-g_{xx})+(g_x^2-g_y^2)g_{yx}}{(g_x^2+g_y^2)^{3/2}}.
$$
\begin{rem}
    In the particular case where $g(x,y)=Ax+By$, i.e., $\Sigma$ is a vertical plane, we have $K_\Sigma^h=H_\Sigma^h=Q_\Sigma^h=0.$ Also, in the case where $g(x,y)=x^2+y^2-R^2$, $H_\Sigma^h=\frac{1}{R}.$
    \end{rem}
 \end{ex} 
\subsubsection{Graphs}
Suppose that $u(x,y,t)=t-f(x,y)$; then we have
    $$ 
    Xu=2y-f_x,\quad Yu=-2x-f_y,\quad Tu=1.
    $$
    Using \eqref{KH-Heis}, \eqref{HH-Heis}  and \eqref{SD-Heis}, respectively, we have:
    \begin{eqnarray}
      \label{K-Heis} K_\Sigma^h&=&4\frac{(f_y+2x)(f_x-2y)(f_{xx}-f_{yy})+(f_y+2x)^2(f_{xy}-2)-(f_x-2y)^2(f_{xy}+2)}{((f_x-2y)^2+(f_y+2x)^2)^2}, \\ 
    \label{H-Heis}  H_\Sigma^h&=&
-\frac{(f_y+2x)^2f_{xx}-2(f_y+2x)(f_x-2y)f_{xy}+(f_x-2y)^2f_{yy}}{((f_x-2y)^2+(f_y+2x)^2)^{3/2}},\\
\label{Q-Heis} Q_\Sigma^h&=&
\frac{(f_y+2x)(f_x-2y)(f_{xx}-f_{yy})+(f_y+2x)^2(f_{yx}-2)-(f_x-2y)^2(f_{xy}+2)}{((f_x-2y)^2+(f_y+2x)^2)^{3/2}}. 
\end{eqnarray}
     \begin{ex}
        {\bf Planes.} Any plane in $\mH$ is an isometric image of either the plane $\Pi_0$ defined by $u(x,y,t)=t=0$ (if it is non-vertical), or the plane $\Pi_1$ defined by $u(x,y,t)=x=0$ (if it is vertical). In the first case we have 
        $$
        H_{\Pi_0}^h=0,\quad K_{\Pi_0}^h=-\frac{2}{x^2+y^2},\quad Q_{\Pi_0}^h=-\frac{1}{\sqrt{x^2+y^2}},
        $$
        whereas in the second case we have
        $$
        H_{\Pi_1}^h=K_{\Pi_1}^h=Q_{\Pi_1}^h=0.
        $$
    \end{ex}
   \begin{ex}
The graph $\Sigma$ of the quadratic surface $\Sigma$  defined by $t=2xy$ satisfies $K_\Sigma^h=H_\Sigma^h=Q_\Sigma^h=0$ and the same hold for all its isometric images.
   \end{ex}
  \subsubsection{Surfaces of revolution}\label{sec-Hsurfrev}
Surfaces of revolution $\Sigma$ are defined by
$$
u(x,y,t)=t-f(r),\quad r=\sqrt{x^2+y^2}.
$$
We have:
$$
Xu=-f'(r)\cos\theta+2r\sin\theta,\quad Yu=
f'(r)\sin\theta+2r\cos\theta,\quad Tu=1;
$$
hence the horizontal velocity is
$
\|\nabla_Hu\|=\sqrt{ f'(r)^2+4r^2}.
$
Observe that  characteristic points exist when $r=0$ and $\dot f(0)=0$, in other words, all such points must lie in the $t$-axis. 
We apply formulae (\ref{K-Heis}),  (\ref{H-Heis}) and(\ref{Q-Heis}), to obtain respectively,
\begin{eqnarray}
   \label{K-rev}
    K_\Sigma^h&=&\frac{4r\frac{d}{dr}(4r^2+ f'(r)^2)-16(4r^2+ f'(r)^2)}{(4r^2+ f'(r)^2)^2}, \\
  \label{H-rev}   H^h_\Sigma& =&-\frac{4r^3  f''(r)+ f'(r)^3}{r (4r^2+ f'(r)^2)^{\frac{3}{2}}},\\
 \label{Q-rev}  Q_\Sigma^h&=&\frac{r\frac{d}{dr}(4r^2+ f'(r)^2)-4(4r^2+ f'(r)^2)}{(4r^2+ f'(r)^2)^{3/2}}.
\end{eqnarray}
 \begin{ex}
    {\bf The Kor\'anyi sphere.}
    Let $S_R$ be the Kor\'anyi sphere centred at the origin and of radius $R$ given by
    $$
    u(x,y,t)=(x^2+y^2)^2+t^2=R^4.
    $$
    Then, $f(r)=\sqrt{R^4-r^4}$ and
     $$
     K_{S_R}^h =\frac{6r^4-2R^4}{r^2R^4},\quad Q_{S_R}^h =\frac{3 r^4-R^4}{r R^2 \sqrt{R^4-r^4}},\quad 
     H_{S_R}^h =\frac{3r}{R^2}.
     $$
    \end{ex}
    \begin{ex}\label{cc-bubble}
      {\bf CC-sphere and bubble set.}  For this example, we refer to \cite{CDPT}; the reader should mind the different conventions of notation. The projections of the shortest horizontal curves (CC-geodesics) that join the origin $O$ of the Heisenberg group and an arbitrary point $p=(z,t)$ are solutions of the isoperimetric problem in the complex plane, that is, they are lifts of Euclidean circles of the form
      \begin{equation}\label{H-geod}
          \gamma_{k,\phi}(\tau)=\left(\frac{1-e^{ik\tau}}{k}e^{i\phi},\,\frac{2}{k^2}(\sin(k\tau)-k\tau)\right),
      \end{equation}
      where $k\in\R$, $\phi\in [0,2\pi)$ and $\tau$ lies in any integral of length $2\pi/|k|$. For a point $p$ such that $z\neq 0$, $\gamma_{k,\phi}$ is unique; in the particular case where $k=0$, the curve $\gamma_{k\phi}$ is a straight line through the origin.

      Setting $\tau=R>0$ in (\ref{H-geod}), the {\it Carnot-Carath\'eodory sphere $S^R_{cc}=S_{cc}(O,R)$ centred at $0$ and of radius $R$} is given by the surface patch
        $$
        \sigma(k, \phi) =
\left(\frac{1-e^{ikR}}{k}e^{i\phi},\, \frac{2}{k^2}(\sin(k R)- kR)\right),
$$
with $(k,\phi)\in(-2\pi/R,2\pi/R)\times(0,2\pi)$. We write
$$
r=A(k):=\frac{2\sin(kR/2)}{k},\quad t=B(k):=\frac{2}{k^2}(\sin(kR)-kR).
$$
Then $t=f(r)=(B\circ A^{-1})(r)$ and we calculate
\begin{eqnarray*}
 f'(r)&=&\frac{dB}{dk}(k)\cdot\frac{dA^{-1}}{dr}(r)=\frac{\frac{dB}{dk}(k)}{\frac{dA}{dk}(k)}\\
 &=&\frac{2}{k}\cdot\frac{(kR)(\cos(kR)+1)-2\sin(kR)}{(kR)\cos(kR/2)-2\sin(kR/2)} =\frac{4}{k}\cos (kR/2),
\end{eqnarray*}
and $(f')^2+4r^2=16/k^2$. 
Therefore
$$
K_{S^R_{cc}}^h=-k^2\frac{(kR)\cos(kR/2)-\sin(kR/2)}{(kR)\cos(kR/2)-2\sin(kR/2)} .
$$
Moreover,
$$
 f''(r)=-\frac{2 \left(k R \sin \left(\frac{k R}{2}\right)+2 \cos
   \left(\frac{k R}{2}\right)\right)}{k^2},
$$
from where we obtain
\begin{eqnarray*}
   H_{S^R_{cc}}^h &=& \frac{\csc ^2\left(\frac{k R}{2}\right) (3 k R+6 \sin (k R)+\sin (2 k R)+k R (\cos (2 k R)-12 \cos (k R)))}{16
   \frac{1}{|k|} \left(k R \cot \left(\frac{k R}{2}\right)-2\right)},
\end{eqnarray*} 
and , from Proposition \ref{prop-K-Q-ac} , we obtain
\begin{eqnarray*}
    Q_{S^R_{cc}}^h &=& \frac{\|\nabla_H u\|}{4} K_{S^R_{cc}}^h.
   \end{eqnarray*}
    We next set $s=k\tau$ and $R=1/|k|$ in (\ref{H-geod}); we also apply the vertical translation $(z,t)\mapsto (z, t+2\pi R^2)$. The {\it bubble set $\mathcal{B}_R=\mathcal{B}(O,R)$ centred at $O$ and of radius $R>0$}  is the surface of revolution defined by the surface patch
        $$
        \sigma_R(s, \phi) = 2R\left(\sin(s/2R)e^{i\phi},\, R \sin(s/R)-s+\pi R
\right),\quad (s, \phi) \in(0, 2\pi R) \times (0, 2\pi).
        $$
        We have
        $$
        r=|z(s)|=2R\sin(s/2R),\quad t=2R(R\sin(s/R)-s+\pi R).
        $$
        Eliminating $s$ we have the formula
        $$
        t=f(r)=r\sqrt{4R^2-r^2}-4R^2\arcsin(r/(2R))+2\pi R^2.
        $$
        Straightforward calculations then deduce
        $$
     K_{\mathcal{B}_R}^h =\frac{1}{R^2}-\frac{2}{r^2},\quad Q_{\mathcal{B}_R}^h =\frac{Rr}{\sqrt{4R^2-r^2}}\left(\frac{1}{R^2}-\frac{2}{r^2}\right),\quad 
     H_{\mathcal{B}_R}^h =\frac{1}{R}.
     $$
     Note that in this particular case, the horizontal mean curvature may be computed using Proposition \ref{k-H}: in fact, it is equal to the signed curvature of the circles
     $$
     z_\phi(s)=2R\sin(s/2R)e^{i\phi}.
     $$
    \end{ex}
    \begin{prop}\label{HKineq-Heis}
For surfaces of revolution in the Heisenberg group, the following sharp inequality holds:
$$
(H_\Sigma^h)^2-K_\Sigma^h>0. 
$$
\end{prop}
\begin{proof}
Straightforward calculations deduce the following equality:
$$
(H_\Sigma^h)^2-K_\Sigma^h=\frac{16r^6( f''(r)- f'(r)/r)^2+16r^2( f'(r)^4+  5 f'(r)^2+8r^4)+ f'(r)^6}{r^2(4r^2+ f'(r)^2)^3}>0.
$$ 
\end{proof}
\subsubsection{Constant horizontal Gauss curvature of surfaces of revolution}
\begin{thm}\label{thm:H-constK}
    The surfaces of revolution $\Sigma =\{(t,r):t=f(r)\}$ in the Heisenberg group that have constant horizontal Gaussian curvature $K_\Sigma^h=k$ are given by
    \begin{enumerate}
        \item the two parameter family of surfaces \begin{equation}\label{0SOR}
 (t-C')^2=\frac{1}{9C^2}(Cr^2-4)^3,\quad Cr^2-4>0,
\end{equation} if $k=0$,
\item the two parameter family of surfaces, written as an integral,
\begin{eqnarray*}
    f(r)=\pm \int_{r_0}^r 2s\sqrt{\frac{\pm s^4+4s^2-C}{C\mp s^4}} ds,\quad C\mp r^4>0,\,\pm r^4+4r^2-C>0,
\end{eqnarray*}
for a constant $r_0>0$ if $k=\pm 1$. \\
In the special case of $C=0$ we have the one parameter family of surfaces
\begin{equation}\label{Kneg-c0}
t-C'=\pm r\sqrt{(4-r^2)}\pm 4\arctan\left(\frac{r}{\sqrt{4-r2}}\right).
\end{equation}
    \end{enumerate}
\end{thm}
\begin{proof}
Let $k\in\R$; our aim is to find all surfaces of revolution satisfying $K_\Sigma^h=k$. This is equivalent to the ODE
\begin{equation}\label{constKH}
  r f' f''-2 (f')^2-4r^2=\frac{k}{8}(4r^2+ (f')^2)^2.  
\end{equation}
Due to Proposition \ref{K-HHinvariance} we may normalise so that $k=-1,0,1$.
By setting $g= f'$,  ODE \ref{constKH} simplifies to
\begin{equation}\label{constKHr}
  rg\frac{dg}{dr}-2g^2-4r^2= \frac{k}{8}(4r^2+g^2)^2. 
\end{equation}
\noindent{\it Case I: $k=0$.} Equation (\ref{constKHr}) becomes
$$
rg\frac{dg}{dr}-2g^2-4r^2=0\iff \frac{dg}{dr}-\frac{2g}{r}=\frac{4r}{g}.
$$
This is a Bernoulli equation whose general solution is
$$
g(r)=\pm r\sqrt{Cr^2-4},\quad C\in\R,\, Cr^2-4>0.
$$
Therefore
$$
f(r)=\frac{1}{3C}(Cr^2-4)^{3/2}+C',\quad C'\in\R.
$$
Hence surfaces of revolution with zero horizontal Gauss curvature are within the two parameter family of surfaces given by \eqref{0SOR}. 

\medskip

\noindent{\it Case II: $k=\pm 1$.} We set $G=4r^2+g^2$ and we have the equation
$$
\frac{dG}{dr}-\frac{4G}{r}=\pm\frac{G^2}{4r}.
$$
This is a Bernoulli equation whose solution is
$$
G(r)=\left\{\begin{matrix}
    \frac{16r^4}{C-r^4}& &k=1,\, C\in\R,\, C-r^4>0,\\
    \\
     \frac{16r^4}{C+r^4}& &k=-1,\, C\in\R,\, C+r^4>0.
\end{matrix}\right.
$$
This gives
$$
g^2(r)=4r^2\frac{\pm r^4+4r^2-C}{C\mp r^4},\quad \pm r^4+4r^2-C>0.
$$
Therefore
$$
\frac{df}{dr}=\pm 2r\sqrt{\frac{\pm r^4+4r^2-C}{C\mp r^4}},\quad C\mp r^4>0,\,\pm r^4+4r^2-C>0.
$$
In the special case when $C=0$ ($k=-1$) the previous equation becomes
$$
\frac{df}{ds}=\pm2r\sqrt{\frac{4}{r^2}-1},\quad r^2<16.
$$
Equivalently, 
$$
\frac{df}{dr}=\pm\sqrt{4-r^2},\quad r<4,
$$
from where we obtain
\begin{equation*}
\pm f(r)=\sqrt{4-r^2} r+4 \arctan\left(\frac{r}{\sqrt{4-r^2}}\right)+C',
\end{equation*}
and the surfaces \eqref{Kneg-c0}.
When $C\neq 0$, the solution to this ODE is an elliptic integral. 
\end{proof}
\subsubsection{Constant horizontal mean curvature of surfaces of revolution}
\begin{thm}\label{thm:H-constH}
    The surfaces of revolution $\Sigma =\{(t,r):t=f(r)\}$ in the Heisenberg group that have constant horizontal mean curvature $H_\Sigma^h=h$ are given by
    \begin{enumerate}
        \item the two parameter family of surfaces given by
        \begin{equation}\label{HMC0-SOR-H}
 (t-C')^2= (C/8)^2\left(16r^2-C^2\right),  
\end{equation}
if $h=0$.
\item the two parameter family of surfaces given by \begin{eqnarray}\label{hmkcheis}
t-C'&=&\mp\frac{1}{2h}\sqrt{(8/h)(s-C)-s^2}\pm\frac{2}{h^2}\arcsin\left(\frac{s-4/h}{\sqrt{16/h^2-8C/h}}\right),
\end{eqnarray}
for $s=2hr^2+C$ $16/h^2-8C/h>0$ and $(8/h)(s-C)-s^2>0$, if $h\neq 0$.
    \end{enumerate}
\end{thm}
\begin{proof}
If $h\in\R$ is a constant, we set 
$$
g=\frac{1}{ f'(r)^2}+\frac{1}{4r^2},
$$ 
and we obtain the ODE
$$\frac{dg}{dr}=4hrg^{3/2}.$$
By separating the variables, this gives
$$g=\frac{4}{(2hr^2+C)^2}$$
where $C\in\R$ and $2hr^2>-C$. In this way, we deduce
$$
\frac{1}{ f'(r)^2}+\frac{1}{4r^2}=\frac{4}{(2hr^2+C)^2}\implies  f'(r)=\pm\frac{2r(2hr^2+C)}{\sqrt{16r^2-(2hr^2+C)^2}},
$$
where $16r^2-(2hr^2+C)^2>0$.
If $h=0$,
$$
 f'(r)=\pm\frac{2Cr}{16r^2-C^2},
$$
which gives
$$
f(r)=\pm\frac{C}{8}\sqrt{16r^2-C^2}+C',\quad r^2>C^2/16.
$$
Therefore the zero horizontal mean curvature surfaces (or, {\it horizontally minimal surfaces} of revolution are given by the 2-parametric family of hyperboloids \eqref{HMC0-SOR-H}. 

If $h\neq 0$, by changing the variable in $s=2hr^2+C$ we obtain the differential equation
$$
\frac{df}{ds}=\pm\frac{1}{2h}\frac{s}{\sqrt{(8/h)(s-c)-s^2}}.
$$
The solution is
$$
f(s)=\mp\frac{1}{2h}\sqrt{(8/h)(s-c)-s^2}\pm\frac{2}{h^2}\arcsin\left(\frac{s-4/h}{\sqrt{16/h^2-8C/h}}\right)+C',
$$
and we conclude that the surfaces of revolution with constant horizontal mean curvature are within the 2-parameter family of surfaces 
 given by \eqref{hmkcheis}. In the particular case when  $C=0$, $h\neq 0$, from (\ref{hmkcheis}) 
we obtain the 1-parameter family of surfaces
\begin{equation}\label{Hmcc0}
t-C'=\mp\frac{r}{h}\sqrt{4-h^2r^2}\pm\frac{2}{h^2}\arcsin\left(\frac{h^2r^2-2}{2}\right).
\end{equation}
\end{proof}
\begin{rem}
    Note that for $h=1/R$, $R>0$, (\ref{Hmcc0}) becomes
    $$
    t-C'=\mp r\sqrt{4R^2-r^2}\mp 2R^2\arcsin\left(\frac{r^2}{2R^2}-1\right).
    $$
    To prove that this gives the bubble set, it suffices to prove the formula
    $$
    \arcsin\left(\frac{r^2}{2R^2}-1\right)=2\arcsin\left(\frac{r}{2R}\right)-\frac{\pi}{2}.
    $$
    Indeed, let $\arcsin(r/(2R)=\theta$.
    Then,
    \begin{eqnarray*}
    \arcsin\left(\frac{r^2}{2R^2}-1\right)&=& \arcsin(2\sin^2\theta-1) \\
    &=&\arcsin(-\cos(2\theta))\\
    &=&\arcsin(\sin(2\theta-\pi/2))\\
    &=&2\theta-\frac{\pi}{2},
    \end{eqnarray*}
    and our claim is proved.
\end{rem}
\subsubsection{Constant symplectic distortion of surfaces of revolution}
\begin{thm}\label{thm:H-constQ}
    The surfaces of revolution $\Sigma =\{(t,r):t=f(r)\}$ in the Heisenberg group that have constant symplectic distortion $Q_\Sigma^h=q$ are given by
    \begin{enumerate}
        \item \eqref{0SOR} if $q=0$;
        \item the two parameter family of surfaces defined by \begin{eqnarray*}
            f(r) &=& \int^r_{r_0}  2s \tan \left( \pm \arccos \left(-\frac{q s}{2}-\frac{c_1}{s}\right) \right)\;ds,\
        \end{eqnarray*}
         where $r_0>0$ is a constant, if $q\neq 0$.
    \end{enumerate}
    \end{thm}
    \begin{proof}
Let $q\in\R$ be a constant; we will solve $Q^h_\Sigma=q$. If $q=0$, we use corollary \ref{ZeroQ-K-Cor} to obtain \eqref{0SOR}. If $q\neq 0$, we set 
$g(r)=f'(r)$ and \eqref{Q-rev} becomes
$$
     Q_\Sigma^h=\frac{r\frac{d}{dr}(4r^2+ g(r)^2)-4(4r^2+ g(r)^2)}{(4r^2+ g(r)^2)^{3/2}}.
$$
Using the substitution $ g(r) = 2r \tan \phi (r)$,
we write the ODE $Q_\Sigma^h =q$ equivalently as
$$
    \sin \phi (r) \phi '(r) -\frac{\cos \phi (r)}{r} = q,
    $$ 
    which implies
    $$
    \phi (r)=\pm \arccos \left(-\frac{q r}{2}-\frac{c_1}{r}\right),
$$
for a constant $c_1.$ \\
This yields 
$$
    f'(r)= 2r \tan \left( \pm \arccos \left(-\frac{q r}{2}-\frac{c_1}{r}\right) \right),
$$
and our claim is proved.
    \end{proof}
\subsection{Affine-additive group}\label{Sec:ex-AA}
 The \textit{affine-additive group} $\mathcal{AA}$ (for a detailed description, see \cite{B-Thesis}), is defined as $\mathcal{AA}:=\mathbb{R} \times \mathbf{H}^1_{\mathbb{C}}$, where $\mathbf{H}^1_\C:=\{(\lambda,t):\lambda>0,t\in\R\}$ is the right half-plane model for the hyperebolic plane. The group operation is given by 
 $$
 p \cdot p'=(a+a',\lambda \lambda' , \lambda t'+t),
 $$ 
 for every $p=(a,\lambda,t)$ and $p'=(a',\lambda',t')$.
An orthonormal basis for the tangent bundle $\rm{T}\,\mathcal{AA}$ comprises the left-invariant vector fields
$$V=2\lambda \partial_\lambda, \, U=\partial_a+2\lambda \partial_t, \, W =-\partial_a.$$
The only non-vanishing Lie bracket relation is $[V,U]=2(U+W)$ and, following the notation from Proposition \ref{Prop: abc}, we have $b_3=c=2$ and $a_i=a_3=b_i=0$, $i=1,2$. 
The contact form is $$\vartheta =\frac{dt}{2\lambda}-da,$$ so that $V,U\in\ker\vartheta$ and $W$ is the Reeb vector field. The corresponding coframe comprise the 1-forms $$\omega_1=\frac{dt}{2\lambda}, \; \omega_2=\frac{d\lambda}{2\lambda}, \; \vartheta=\frac{dt}{2\lambda}-da.$$
The sectional curvatures of the characteristic planes are calculated by using \eqref{SecCurvXY},\eqref{SecCurvXT} and \eqref{SecCurvYT}; that is, 
\begin{equation*}
K_\epsilon(X,Y)=-4-\frac{3}{\epsilon^2}\quad
K_\epsilon(X,T_\epsilon)=K_\epsilon(Y,T_\epsilon)=1/\epsilon^2.
\end{equation*}
The isometry group $\rm{Isom}_\mathcal{H}(\mathcal{AA})$ is just $\Aa\times\mathbb{Z}_2$: that is, it is generated by left translations $L_{(a_0 ,\lambda_0, t_0)}:\mathcal{AA}\to\mathcal{AA}$ with respect to $(a_0 ,\lambda_0, t_0)\in\Aa$ and
the conjugation $j:\Aa \to \Aa$ given by 
\begin{eqnarray*}
j(a ,\lambda ,t) =(-a ,\lambda ,-t),   
\end{eqnarray*}
(see Section 6.2 in \cite{H-ISR}).
\subsubsection{Surfaces embedded in the affine-additive group}\label{Sec:AA-surf}
Let $\Sigma$ be a regular hypersurface of $\Aa$ defined by the equation $u=0$. Then the horizontal Gauss curvature $K_\Sigma^h$ of $\Sigma$  is given by 
\begin{eqnarray*}
    K_\Sigma^h &=& 2E_1\left( \frac{Wu}{\|\nabla_H u\|}  \right) -4\left( \frac{Wu}{\|\nabla_H u\|} \right)^2,
\end{eqnarray*}
for $E_1 =-\overline{q} V+\overline{p}U$. The horizontal mean curvature $H_\Sigma^h$ of $\Sigma$ is given by
\begin{equation*}
  H_\Sigma^h=V\left(\frac{Vu}{\|\nabla_Hu\|}\right) +U\left(\frac{Uu}{\|\nabla_Hu\|}\right)-2\frac{Vu}{\|\nabla_H u\|}. 
\end{equation*}
and the symplectic distortion $Q_\Sigma^h$ of $\Sigma$ is given by
\begin{equation*}
  Q_\Sigma^h=V\left(\frac{Uu}{\|\nabla_Hu\|}\right) -U\left(\frac{Vu}{\|\nabla_Hu\|}\right)-2\frac{Uu}{\|\nabla_H u\|}.
\end{equation*}
\begin{ex}\label{ex-cyl-AA}
     {\bf $\Aa$-Cylindrical surfaces.}
     $\Aa$-Cylindrical surfaces $\Sigma$ are defined by equations of the form $u(a,\lambda,t)=g(\lambda,t)=0$. These surfaces are invariant by translations of the form
     $$
     (a,\lambda,t)\mapsto (a'+a,\lambda,t),\quad a'\in\R.
     $$
    The horizontal Gauss curvature of all these surfaces is zero. Moreover,
$$
H_\Sigma^h=2\frac{ g_{\lambda}^2( \lambda  g_{tt}-g_\lambda)+g_{t}^2
   \left( \lambda  g_{\lambda \lambda}- g_{\lambda}\right)-2 \lambda  g_{t}g_{\lambda t} g_{\lambda} }{\left(g_{\lambda}^2+g_{t}^2\right)^{3/2}},
$$
and
$$
Q_\Sigma^h=2\frac{g_\lambda^2(\lambda g_{\lambda t}-g_t)-g_t^2(\lambda g_{\lambda t}+g_t)+\lambda g_\lambda g_t  (g_{t t}-g_{\lambda\lambda})}{(g_\lambda^2+g_t^2)^{3/2}},
$$
\begin{rem}
    When $\Sigma$ is a vertical plane, $g(\lambda,t)=A\lambda+Bt$  and we have $H_\Sigma^h=-\frac{2 A}{\sqrt{A^2+B^2}}$ and $Q_\Sigma^h=-\frac{2 B}{\sqrt{A^2+B^2}}.$ In the case when $g(\lambda,t)=(\lambda-\lambda_0)^2+(t-t_0)^2-R^2$, we have $H_\Sigma^h=\frac{2 \lambda_0}{R}$, $Q_\Sigma^h=-\frac{2 (t-t_0)}{R}$.
    \end{rem}
 \end{ex} 
\subsubsection{Invariance}
The following holds:
\begin{prop}
    The horizontal Gaussian curvature $K_\Sigma^h$ of a regular surface $\Sigma$ embedded in $\Aa$ is invariant under all maps of the form $h(a,\lambda ,t) =(a+a_0, \delta \lambda , \pm\delta t+t_0)$, where $\delta>0$ and $a_0,t_0\in\R$.
\end{prop}
 \subsubsection{Graphs}\label{Sec-AA-graph}
Let $f:\bH_\C^1$ be a smooth function and consider regular surfaces $\Sigma$ defined by 
$$
\Sigma=\{(a,\lambda,t)\in\Aa:u(a,\lambda,t)=a-f(\lambda,t)=0\}.
$$  
Since
\begin{equation*}
    Vu = -2\lambda f_\lambda, \quad
    Uu = 1-2\lambda f_t,\quad
    Wu=-1,    \end{equation*}
then according to the formulae in Section \ref{Sec:AA-surf} we  obtain:  
\begin{eqnarray*}
       K^h_{\Sigma}&=&\frac{4}{(4\lambda^2f_\lambda^2+(2\lambda f_t-1)^2)^2}\left\{\left(2\lambda(f_t+\lambda f_{t\lambda})-1\right)(2\lambda f_t-1)^2\right.\label{K-graphAA}\\
      &&\left.+ 4\lambda^2 f_\lambda \left(f_\lambda +\lambda(f_{\lambda \lambda} -f_{tt})\right)(2\lambda f_t-1)\right.\nonumber\\
&&\left.-4\lambda^2f_\lambda^2(1+2\lambda^2 f_{\lambda t})\right\}\nonumber, \\
         H_{\Sigma}^h &=& 4\lambda \frac{-\lambda f_{\lambda \lambda}(2\lambda f_t-1)^2+2\lambda f_\lambda (2\lambda f_t-1) (f_t+2\lambda f_{\lambda t})+4\lambda^2 f_\lambda ^2(f_\lambda -\lambda f_{tt} )}{\left(4\lambda^2 f_\lambda^2+(2\lambda f_t-1)^2\right)^{3/2}},\nonumber \\ 
Q_\Sigma^h &=& \frac{1}{\left(4 \lambda ^2 f_{\lambda}^2+\left(1-2 \lambda  f_{t}\right)^2\right)^{3/2}} \{ 4 \lambda  \left(\left(4 \lambda ^3 f_{\lambda t}-6 \lambda
   \right) f_{t}^2+4 \lambda ^2 f_{t}^3 \right. \nonumber\\
   && \left.+ \left(4 \lambda ^2 \left(f_{\lambda} \left(-\lambda 
   f_{tt}+f_{\lambda}+\lambda  f_{\lambda \lambda}\right)-f_{\lambda t}\right)+3\right) f_{t} \right. \nonumber \\
   && \left.+ \lambda  \left(f_{\lambda t}-2 f_{\lambda}
   \left(2 f_{\lambda} \left(\lambda ^2 f_{\lambda t}+1\right)-\lambda  f_{tt}+\lambda  f_{\lambda\lambda}\right)\right)\right)-2 \}. \nonumber
\end{eqnarray*}
\begin{ex}{\bf Planes.}
Let $A,B,C,D\in\R$ and consider 
$$
\Pi=\{(a,\lambda,t)\in\Aa:Aa+B\lambda+Ct+D=0\}.
$$
By applying the above formulae we have
$$
K^h_{\Pi}=-\frac{4 A (A^3 + 6 A^2 C \lambda + 4 A \lambda^2 (2 B^2 + 3 C^2) + 8 C \lambda^3 (B^2 + C^2))}{(A^2 + 4 A C \lambda + 4 \lambda^2 (B^2 + C^2))^2}.
$$
    It is evident that planes with $A=0$ have zero horizontal Gaussian curvature and planes with $C=0$ have negative Gaussian curvature equal to 
$$
K^h_{\Pi}=-\frac{4(A^4+8A^2 B^2\lambda^2)}{(A^2+4 B^2\lambda^2)^2}.
$$
For the horizontal mean curvature we have the formula
\begin{eqnarray*}
    H_{\Pi}^h &=& -\frac{8 B \lambda^2 \left(A C+2 \lambda
   \left(B^2+C^2\right)\right)}{\left((A+2 C \lambda)^2+4 B^2
   \lambda^2\right)^{3/2}}.  
\end{eqnarray*}
Finally, the symplectic distortion is given by
$$Q^h_{\Pi}=-\frac{2 \left(A^3+6 A^2 C \lambda +4 A \lambda ^2 \left(2 B^2+3
  C^2\right)+8 C \lambda ^3
   \left(B^2+C^2\right)\right)}{\left((A+2 C \lambda
   )^2+4 B^2 \lambda ^2\right)^{3/2}}.
   $$
   In particular, the hyperbolic plane $\mathbf H^1_\C=\{(a,\lambda,t)\in\Aa:a=0\}$ has constant negative Gaussian curvature equal to $-4$, zero horizontal mean curvature and constant symplectic distortion equal to -2.   
\end{ex}
\begin{ex}\label{ex-cyl2-AA} 
Let $f:\bH_\C^1\to\R$ be a smooth function with non-vanishing gradient. We consider the surface $$\Sigma=\{(a,\lambda,t)\in\Aa:f(a,\lambda)=0\}.$$ 
Then \begin{eqnarray*}
K_{\Sigma}^h&=&-\frac{4 \left(4 \lambda^3 f_{\lambda} f_{\lambda \lambda} f_{a}^2+4
  \lambda^3 f_{\lambda}^3 f_{aa}+8 \lambda^2 f_{\lambda}^2
   \left(f_{a}-\lambda f_{a\lambda}\right)
   f_{a}+f_{a}^4\right)}{\left(4 \lambda^2
   f_{\lambda}^2+f_{a}^2\right)^2}, \\
  H_\Sigma^h &=&\frac{4 \lambda^2 \left(f_{\lambda \lambda} f_{a}^2+f_{\lambda} \left(f_{\lambda}
   \left(f_{aa}-4 \lambda  f_{\lambda}\right)-2
   f_{a} f_{a \lambda}\right)\right)}{\left(4
   \lambda ^2 f_{\lambda}^2+f_{a}^2\right)^{3/2}},\\
   Q_\Sigma^h &=& \frac{8 \lambda^2 \left(\lambda  f_{a\lambda}-2
   f_{a}\right) f_{\lambda}^2+2 \lambda 
   f_{a} \left(f_{aa}-4 \lambda ^2
   f_{\lambda \lambda}\right) f_{\lambda}-2
   f_{a}^2 \left(f_{a}+\lambda 
   f_{a\lambda}\right)}{\left(4 \lambda^2 f_{\lambda}^2+f_{a}^2\right)^{3/2}}.
   \end{eqnarray*}
   \end{ex}
\subsubsection{Surfaces of revolution}\label{sec-AAsurfrev}
The group $\Aa$ does not have a {\it canonical} axis of revolution. By choosing a one-parameter subgroup $H$ to act as rotational symmetry, the resulting surfaces of revolution will be the surfaces in $\Aa$ which are invariant under the action of $H$. For a general surface of revolution, we may give the following definition.
\begin{defn}
  Let $H$ be a one-parameter subgroup of $\Aa$. $\Sigma$ is a surface of revolution with respect to $H$ if it is invariant under the left translation of $H$.  
\end{defn}
We distinguish three cases.

\medskip

\noindent{\it Case 1.} Let $H_a=\{ (a,0,0) : a\in \mathbb{R} \}$ with action
    $(a,\lambda,t) \mapsto (a_0+a,\lambda ,t).$
    A {\it surface of revolution around the $a$-axis} is invariant under the action of $H_a$. These surfaces are $\Aa$-cylindrical with base curve in ${\bf H}_{\mathbb{C}}^1$ and which are extended vertically along the $a$-axis, in other words,  $\Sigma : u(\lambda,t)=0$, see Example \ref{ex-cyl-AA}.
        
    \medskip
    
\noindent{\it Case 2.} Let $H_t=\{(0,0,t) : t\in \mathbb{R}\}$ acting by
    $(a_0,\lambda_0,t) \mapsto (a_0,\lambda_0,t+t_0).$
    Again, we have  a $\Aa$-cylindrical surface
    extended vertically along the $a$-axis, in other words $\Sigma : u(a,\lambda)=0$, see Example \ref{ex-cyl2-AA}.

    \medskip
    
    \noindent{\it Case 3.} Let  $H=\R_{>0}$ act by transformations $(a,\lambda,t)\mapsto(a,\delta \lambda,\delta t)$, $\delta>0$. Then the corresponding surfaces of revolution are hypersurfaces $\Sigma$ given by $u(a,t/\lambda)=0$.
It is convenient at this point to consider the change of parameters $\lambda \to \lambda$, $t \to \rho\lambda$, $a \to a$; then the frame $V,U,W$ becomes
$$
V = -2 \rho\partial_\rho +2\lambda\partial_\lambda,\quad
    U = 2\partial_\rho+\partial_a,\quad
    W = -\partial_a.
$$
We now consider graphs defined by $u(a,\rho)=a-f(\rho)=0$  and we have:
$$
VP=-4\rho( f'(\rho)+\rho f''(\rho)),\quad VQ=-2\rho f''(\rho),\quad UP=4( f'(\rho)+\rho f''(\rho)),\quad UQ=4 f''(\rho),
$$
and from the discussion in Section \ref{Sec-AA-graph} it follows that
the horizontal Gaussian curvature is given by
\begin{eqnarray} \label{HGausscurvtlAA}
    K^h_\Sigma &=& 4\frac{2\rho f''(\rho)(2(\rho^2+1)f'(\rho)-1)-4(f'(\rho))^2+4f'(\rho)-1
    }{(4(\rho^2+1)f'(\rho)^2-4f'(\rho)+1)^2},
    \end{eqnarray}
    the horizontal mean curvature becomes
\begin{eqnarray} \label{H_f}
    H^h_\Sigma &=& -4\rho \frac{\rho f''(\rho)+4(\rho^2+1)f'(\rho)^3-6(f'(\rho))^2+2f'(\rho)}{(4(\rho^2+1)f'(\rho)^2-4f'(\rho)+1)^{3/2}}
    \end{eqnarray}
    and the symplectic distortion is \begin{eqnarray}\label{Q-rev-AA}
        Q_\Sigma^h &=&2 \frac{2 \rho  f''(\rho )(2(\rho^2+1)f'(\rho)-1)-4 f'(\rho )^2+4 f'(\rho )-1}{\left(4 \left(\rho ^2+1\right)
   f'(\rho )^2-4 f'(\rho )+1\right)^{3/2}}.
    \end{eqnarray}
 \begin{rem}
     One may verify that the characteristic points are the points that satisfy $\rho=0$ and $f'(0)=\frac{1}{2}$.   
     \end{rem}
\subsubsection{Constant horizontal Gaussian curvature of surfaces of revolution} 
\begin{thm}\label{thm:AA-constK}
    Let $f:\R \to \R$ be a $C^2$ function. The surfaces of revolution $\Sigma =\{(a,\rho):a=f(\rho)\}$ that have constant horizontal Gaussian curvature $K_\Sigma^h=k$ are the following:
    \begin{enumerate}
        \item the two parameter family of surfaces defined by
        \begin{equation}\label{K0-revAA}
        f(\rho) = \frac{1}{2} \left(\pm \sqrt{c \rho^2+c-1}\mp \arctan\left(\sqrt{c
   \rho^2+c-1}\right)+\arctan (\rho)\right)+C,\,\,c>0,\,C\in\R,\rho^2>(1/c)-1,
    \end{equation}
    if $k=0$;
    \item  the two parameter family of surfaces defined by
    \begin{eqnarray}\label{K-4a-graphAa}
   f(\rho) &=& \frac{1}{2}
   \left(\pm \arctan\left(\sqrt{\frac{\rho
   ^2+c}{1-c}}\right)+\arctan(\rho )\right)+C,\quad c\in(0,1),\,C\in\R,
\end{eqnarray}
as well as the one parameter family defined by
\begin{eqnarray}\label{K-4b-graphAa} 
f(\rho)&=&\frac{1}{2}\arctan(\rho)+C,\quad C\in\R,
\end{eqnarray}
if $k=-4$;
\item the two parameter family defined by
\begin{equation}\label{K-0-4a-graphAa}
    f(\rho) =\frac{1}{2}\arctan(\rho)\mp\frac{1}{2}\arctan (w)\pm\left\{\begin{matrix}
\dfrac{\sqrt{c_1}}{2} \arctan\left( \dfrac{w}{\sqrt{c_1}} \right) + C, & c_1 > 0 \\
\dfrac{c_1}{4\sqrt{-c_1}} \ln\left| \dfrac{w - \sqrt{-c_1}}{w + \sqrt{-c_1}} \right| + C, & c_1 < 0,
\end{matrix}\right.,
\end{equation}
where $C\in\R$, $c_1=\frac{k+4}{k}$,
$$
w=\sqrt{\frac{\left(1+\frac{k}{4}\right)\rho^2+1-c}{c-\frac{k}{4}\rho^2}}
$$
and $c\in\R$ satisfying
\begin{equation}\label{k-cond-1}
c-\frac{k}{4}\rho^2>0,\quad \left(1+\frac{k}{4}\right)\rho^2>c-1,\quad c\neq k/4,   
\end{equation}
if $k\neq 0.$ In the particular case where $c=-k/4\neq 0$ we have
\begin{equation}\label{k-part}
f(\rho)=\pm\frac{1}{4}\sqrt{-c_1}\ln(\rho^2+1)^{1/4}+\frac{1}{2}\arctan\rho+C ,
\end{equation}
for $\rho \neq 0$.
\end{enumerate}

\end{thm}

\begin{proof}
We put $f'(\rho)=g(\rho)$ in (\ref{HGausscurvtlAA}) to obtain
$$
K_\Sigma^h=4\frac{2\rho(2(\rho^2+1)g(\rho)-1)g'(\rho)-(2g(\rho)-1)^2}{(4\rho^2g^2(\rho)+(2g(\rho)-1)^2)^2}.
$$
Observe that if
$$
2(\rho^2+1)g(\rho)-1=0\iff g(\rho)=\frac{1}{2(\rho^2+1)},
$$
then $K_\Sigma^h=-4$ and therefore we obtain (\ref{K-4b-graphAa})
as a partial solution of the ODE $K_\Sigma^h=-4$.
    Let now $k \in \mathbb{R}$ be a constant and  assume that $h(\rho)=2(\rho^2+1)g(\rho)-1\neq 0$. Then
    $$
    K_\Sigma^h=4\frac{\rho(\rho^2+1)hh'-(h^2+\rho^2)^2+h^4-h^2}{(h^2+\rho^2)^2}.
    $$
   Set next $z=h^2+\rho^2$. Then we have the Bernoulli equation
   $$
   \frac{dz}{d\rho}-\frac{2(2\rho^2+1)}{\rho(\rho^2+1)}z=\frac{k}{2\rho(\rho^2+1)}z^2.
   $$
   When $k=0$,
$$
z=c\rho^2(\rho^2+1), \,\text{and}\,c>0
$$
since $z>0$. Next,
$$
h^2=z-\rho^2=\rho^2(c(\rho^2+1)-1)>0,
$$
therefore $\rho^2>\frac{1}{c}-1.$
We thus have $$h=\pm\rho\sqrt{c(\rho^2+1)-1}\implies f'(\rho)=\frac{\pm\rho\sqrt{c(\rho^2+1)-1}+1}{2(\rho^2+1)}.
$$
The solution to the above ODE is (\ref{K0-revAA}).

\medskip

When $k\neq 0$, the solution to the Bernoulli equation is
   $$
   z=\frac{\rho^2(\rho^2+1)}{c-\frac{k}{4}\rho^2},\quad c\in\R.
   $$
   Since $z\ge 0$ we must have $\frac{k}{4}\rho^2<c$. Also, since
   $$
   h^2=z-\rho^2=\rho^2\frac{\left(1+\frac{k}{4}\right)\rho^2+1-c}{c-\frac{k}{4}\rho^2},
   $$
we must also have $\left(1+\frac{k}{4}\right)\rho^2>c-1.$
We eventually have the ODE
$$
f'(\rho)=\pm\frac{\rho}{2(\rho^2+1)}\sqrt{\frac{\left(1+\frac{k}{4}\right)\rho^2+1-c}{c-\frac{k}{4}\rho^2}}+\frac{1}{2(\rho^2+1)}.
$$
We first examine the particular case where $k=-4$; this gives (\ref{K-4a-graphAa}). Next, in the case $c=-k/4$ we must have $k\in(-4,0)$ and the ODE becomes
$$
f'(\rho)=\pm\frac{1}{2(\rho^2+1)}\left(\rho\sqrt{-\frac{k+4}{k}}+1\right).
$$
Hence we obtain (\ref{k-part}).
For the general case it suffices to calculate the integral
$$
I=\int\frac{\rho}{2(\rho^2+1)}\sqrt{\frac{\left(1+\frac{k}{4}\right)\rho^2+1-c}{c-\frac{k}{4}\rho^2}}\, d\rho:
$$
Explicitly,
\begin{eqnarray*}
I&=&-\frac{1}{2}\int\frac{dw}{1+w^2}+\frac{1}{2}\frac{k+4}{k}\int\frac{dw}{w^2+\frac{k+4}{k}}\\
&=&-\frac{1}{2}\arctan w+I_1
\end{eqnarray*}
and 
$$
I_1=
\begin{cases}
\dfrac{\sqrt{c_1}}{2} \arctan\left( \dfrac{w}{\sqrt{c_1}} \right) + C, & c_1 > 0 \\
\dfrac{c_1}{4\sqrt{-c_1}} \ln\left| \dfrac{w - \sqrt{-c_1}}{w + \sqrt{-c_1}} \right| + C, & c_1 < 0.
\end{cases}.
$$
Here, $c_1=\frac{k+4}{k}$ and the  solution is (\ref{K-0-4a-graphAa}).
\end{proof}
\subsubsection{Constant horizontal mean curvature of surfaces of revolution}
\begin{thm}\label{thm:AA-constH}
Let $f:\R\to\R$ be a $C^2$ function.
The surfaces $\Sigma=\{(a,\rho):a=f(\rho)\}$ that have constant horizontal mean curvature $H_\Sigma^h=h$ are the following:
\begin{enumerate}
    \item the families defined by \eqref{K-4a-graphAa} and \eqref{K-4b-graphAa},
    if $h=0$.
    \item  The two-parameter family of surfaces defined by
    \begin{eqnarray}\label{H=constAA}
        f(\rho) &=& \frac{\arctan \rho}{2} \pm \frac{1}{2}\arccos \left( \frac{A(\rho)}{\sqrt{\rho^2+1}} \right)+g(\rho),
    \end{eqnarray}
    where
    \begin{eqnarray*}
   g(\rho) =\begin{cases}
         \mp \frac{\sqrt{1-c_1^2+2c_1\rho}}{2c_1}+C_1, & \text{if } h= 2 \\
        \mp \frac{\sqrt{1-c_1^2-2c_1\rho}}{2c_1}+C_2 , & \text{if }  h= -2\\
       \mp \frac{h}{4} \frac{1}{\sqrt{\alpha}} {\rm arcsinh} \left( \frac{\sigma}{\sqrt{\Delta/\alpha}} \right) +c_3 , &  \text{if } \alpha>0 \text{ and } \Delta >0 \\
       \mp \frac{h}{4} \frac{1}{\sqrt{\alpha}} {\rm arcosh }\left( \frac{|\sigma|}{\sqrt{|\Delta|/\alpha}} \right) +c_4, & \text{if } \alpha >0, \Delta <0 \text{ and } |\sigma| >\sqrt{|\Delta|/\alpha} \\
       \mp \frac{h}{4} \frac{1}{\sqrt{\alpha}} \ln \left( |\sigma|  \right) +c_5, & \text{if } \alpha >0, \Delta =0 \text{ and } \sigma \neq 0 \\
       \mp \frac{h}{4} \frac{1}{\sqrt{-\alpha}} \arcsin \left( \frac{\sigma \sqrt{-\alpha}}{\sqrt{\Delta}} \right) +c_6, & \text{if } \alpha <0, \Delta>0 \text{ and } |\sigma | < \sqrt{-\Delta/\alpha}
       \end{cases},
    \end{eqnarray*}   
    \end{enumerate}
 where $A(\rho)=-\frac{h}{2}\rho+c_1$, $\alpha =1-\frac{h^2}{4}$, $\beta =c_1h$, $\gamma =1-c_1^2$, $\sigma =\rho+\frac{\beta}{2\alpha}$ and $\Delta=\gamma -\frac{\beta^2}{4\alpha}$, for $\alpha \neq 0$ and $c_1$ and  $\rho \neq 0$ satisfies
$$
\sqrt{\rho^2+1}>\left|c_1-\frac{h}{2}\rho\right|.
$$
\end{thm}
\begin{proof}
We will solve $H_\Sigma^h =h$, for a constant $h$. We define the function $g:\R\to \R$ given by $g(\rho)=f'(\rho)$ and then (\ref{H_f}) becomes
\begin{equation}\label{HAA-g}
H_\Sigma^h=-4\rho\frac{\rho g'(\rho)+2g(\rho)\left(2(\rho^2+1)g^2(\rho)-3g(\rho)+1\right)}{((2\rho g(\rho))^2+(2g(\rho)-1)^2)^{3/2}}.
\end{equation}
We first distinguish the case
$$
2(\rho^2+1)g(\rho)=1\iff f(\rho)=\frac{1}{2}\arctan(\rho)+c,\,c\in\R.
$$
Then straightforward calculations show that $H_\Sigma^h=0$.
Next, for $2(\rho^2+1)g(\rho)\neq 1$
    we use the substitution $v:\R\to \R$ given by \begin{eqnarray} \label{CMCODEAAdilationSub1}
        v(\rho) &=& \frac{2 \left(\rho ^2+1\right) g(\rho )-1}{\sqrt{(2\rho\, g(\rho))^2 +(2g(\rho)-1)^2}}.
    \end{eqnarray}
Then equation $H_\Sigma^h=h$ becomes \begin{eqnarray*}
        v' &=&-\frac{h}{2}.
    \end{eqnarray*}
   Hence
    \begin{eqnarray} \label{CMCODEveqh}
        v &=& -\frac{h}{2}\rho+c_1,
    \end{eqnarray}
    for a constant $c_1$. We now let
    \begin{eqnarray}
    w(\rho) &=& 2(\rho^2+1)g(\rho)-1,
\end{eqnarray}
    in \ref{CMCODEAAdilationSub1}. Then \ref{CMCODEveqh} becomes 
    \begin{eqnarray*}
    \frac{w(\rho) \sqrt{\rho^2+1}}{\sqrt{w^2(\rho)+\rho^2}} =-\frac{h}{2}\rho+c_1.
\end{eqnarray*}
    Solving for $w$ we get
    $$
    w^2(\rho)=\frac{\left( -\frac{h}{2}\rho+c_1 \right)^2 \rho^2}{\rho^2+1-\left( -\frac{h}{2}\rho+c_1 \right)^2}.
    $$
   This is positive when
   $$\sqrt{\rho^2+1}>\left|c_1-\frac{h}{2}\rho\right|.
   $$
   Then \begin{eqnarray*}
        w(\rho) &=& \pm \sqrt{ \frac{\left( -\frac{h}{2}\rho+c_1 \right)^2 \rho^2}{\rho^2+1-\left( -\frac{h}{2}\rho+c_1 \right)^2}}.
    \end{eqnarray*}
    We have \begin{eqnarray*}
        f'(\rho) &=& \frac{w(\rho)+1}{2(\rho^2+1)}.
    \end{eqnarray*}
We will compute the integral $\int \frac{w(\rho )}{2(\rho^2+1)}d \rho$.
Observe that for $\sqrt{\rho^2+1}>\left|A(\rho) \right|$ we have \
\begin{eqnarray}
    \frac{d}{d\rho} \arccos \frac{A(\rho)}{\sqrt{\rho^2+1}} &=& \frac{\rho A(\rho) -(1+\rho^2)A'(\rho)}{(1+\rho^2) \sqrt{\rho^2+1-A(\rho)^2}} \nonumber\\
    \frac{\rho A(\rho)}{(\rho^2+1)\sqrt{\rho^2+1-A(\rho)^2}} &=& \frac{A'(\rho)}{\sqrt{\rho^2+1-A(\rho)^2}}+ \frac{d}{d \rho} \arccos \left( \frac{A(\rho)}{\sqrt{\rho^2+1}} \right). \label{arccosDerIden}
\end{eqnarray}
Set $A(\rho) = -\frac{h}{2} \rho +c_1.$
We then have
\begin{eqnarray*}
    \int \frac{w(\rho )}{2(\rho^2+1)} &=& \pm \frac{1}{2} \int \frac{\rho A(\rho)}{(\rho^2+1)\sqrt{\rho^2+1-A(\rho)^2} }d\rho \\
    &\overset{\eqref{arccosDerIden}}{=}& \pm \frac{1}{2} \left( \arccos \left( \frac{A(\rho)}{\sqrt{\rho^2+1}}\right) +\int \frac{A'(\rho)}{\sqrt{\rho^2+1-A(\rho)^2}} d\rho +c_2\right).
\end{eqnarray*}
Using the identity \begin{eqnarray*}
    \arccos \left( \frac{1}{\sqrt{1+x^2}} \right) &=& \arctan (x),
\end{eqnarray*} and setting $c_1=\sqrt{1-c}$
we see that the surfaces that have $h=0$ are given by \eqref{K-4a-graphAa} together with \eqref{K-4b-graphAa}, whic we have showed already.

If $h= \pm2$, we obtain $f(\rho)$ as in \eqref{H=constAA}, where $g(\rho)$ is given by the first two cases.

Now we set $\alpha =1-\frac{h^2}{4}$, $\beta =c_1h$, $\gamma =1-c_1^2$ and $\Delta=\gamma -\frac{\beta^2}{4\alpha}$, for $\alpha \neq 0$, then by using the change of variable $\sigma =\rho+\frac{\beta}{2\alpha}$, we obtain
\begin{eqnarray*}
  A'(\rho)  \int \frac{d\rho}{\sqrt{\rho^2+1-A(\rho)^2}} &=& A'(\rho) \int\frac{d\sigma}{\sqrt{\alpha\sigma^2+\Delta}}
\end{eqnarray*}
We will compute \begin{eqnarray}\label{IntQuadsqrt}
    I &=& A'(\rho) \int\frac{d\sigma}{\sqrt{\alpha\sigma^2+\Delta}} .
\end{eqnarray}
We distinguish the following cases:
\begin{itemize}
    \item Let first $\alpha >0$.
\noindent{Case 1. $\Delta >0$ }
\eqref{IntQuadsqrt} becomes
\begin{eqnarray*}
  \frac{A'(\rho )}{\sqrt{\alpha}} {\rm arsinh} \left( \frac{\sigma}{\sqrt{\Delta/\alpha}} \right)+c_3,\text{ for } \sigma \in \R.
\end{eqnarray*}
\noindent{Case 2. $\delta <0$ }
\eqref{IntQuadsqrt} becomes
\begin{eqnarray*}
     \frac{A'(\rho )}{\sqrt{\alpha}} {\rm arcosh} \left( \frac{|\sigma|}{\sqrt{|\Delta|/\alpha}} \right)+c_4,\text{ for } \left|\sigma\right| >\sqrt{|\Delta|/\alpha}.
\end{eqnarray*}

\noindent{Case 3. $\delta =0$ }
\eqref{IntQuadsqrt} becomes
\begin{eqnarray*}
    \frac{A'(\rho )}{\sqrt{\alpha}} \ln \left|\sigma  \right| +c_5, \text {for } \sigma\neq 0.
\end{eqnarray*}

\item  If $\alpha <0$

\noindent{When $\delta <0$ }
\eqref{IntQuadsqrt} becomes
\begin{eqnarray*}
     \frac{A'(\rho )}{\sqrt{-\alpha}} \arcsin \left( \frac{\sigma \sqrt{-\alpha}}{\sqrt{\Delta}} \right) +c_6,
\end{eqnarray*}
for $\left|\sigma\right| <\sqrt{-\Delta/\alpha}$.
\end{itemize}  
Summing up we obtain \eqref{H=constAA}.     
\end{proof}
Conditions \eqref{K-4a-graphAa}, \eqref{K-4b-graphAa} of Theorem \ref{thm:AA-constH} immediately deduce the following.
\begin{cor}\label{AA-SOR-K0-H-4}
A surface of revolution $\Sigma$ of the form $u(a,\rho)=a-f(\rho)=0$ has zero horizontal mean curvature if and only if it has horizontal Gaussian curvature equal to $-4$.
\end{cor}
\begin{figure}[ht!]
\centering
\includegraphics[width=0.82\textwidth]{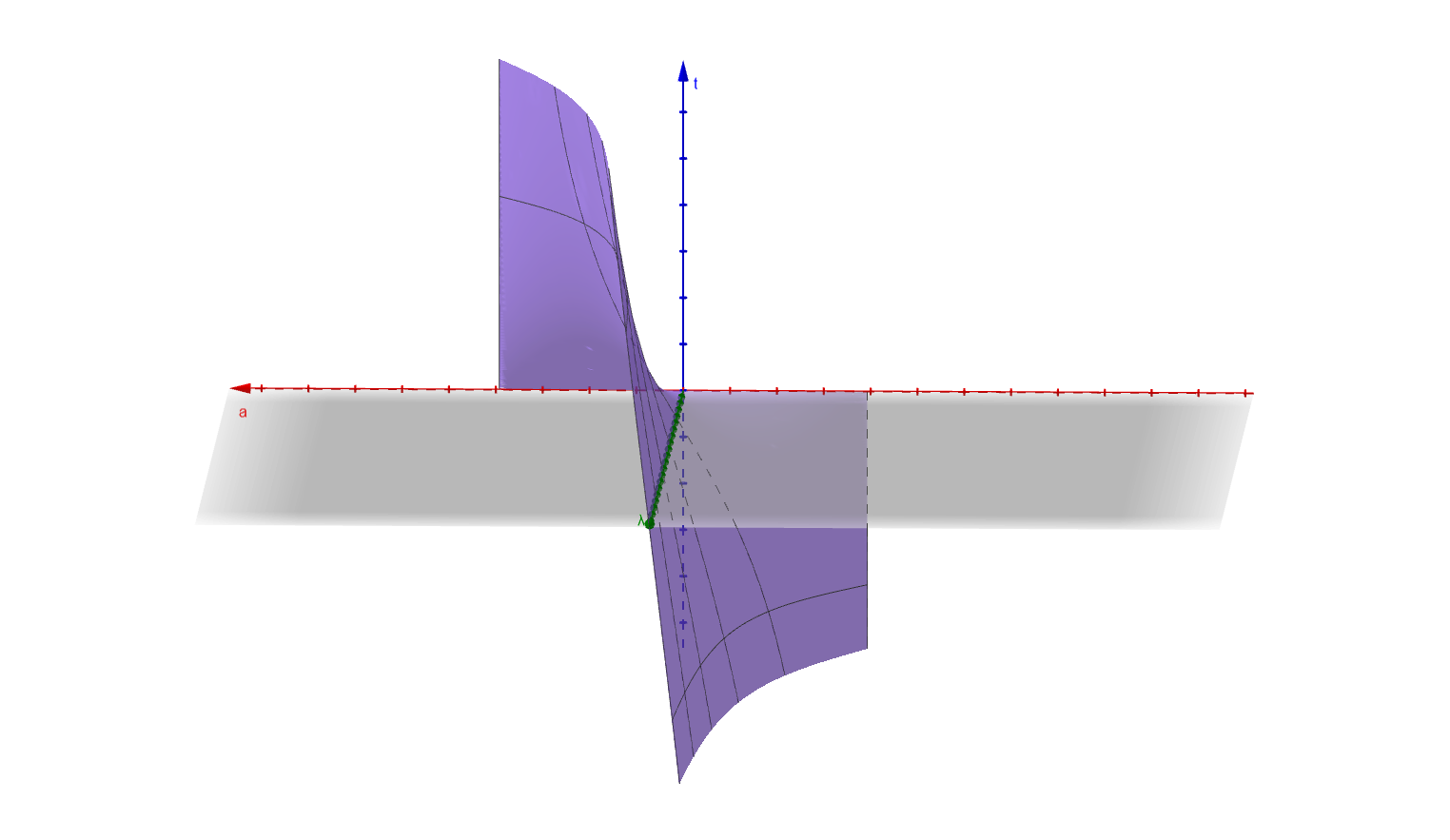}
\caption{Plot of $\Sigma=\{(a,\rho):a=\frac{1}{2}\arctan(\rho)\}$ satifying $H^h_\Sigma=0$ and $K^h_\Sigma=-4$.}
\end{figure}
\newpage
\subsubsection{Constant symplectic distortion of surfaces of revolution}
\begin{thm}\label{thm:AA-constQ}
Let $f:\R\to\R$ be a $C^2$ function.
The surfaces $\Sigma=\{(a,\rho):a=f(\rho)\}$ that have constant symplectic distortion $Q_\Sigma^h=q$ are the following:
\begin{enumerate}
    \item the family defined by \eqref{K0-revAA} if $q=0$;
    \item the two parameter family defined by the integral 
    \begin{eqnarray}
        f(\rho) &=& \int^{\rho}_{\rho_0} \frac{s \left(\left(\frac{c_1}{s}-\frac{q}{2}\right)\pm
   \sqrt{-c_1^2+c_1 q s-\frac{1}{4}
   \left(q^2-4\right) s^2+1}\right)}{\left(s^2+1\right) (2
   c_1-q s)} ds,
    \end{eqnarray}
    for a constant $\rho_0 \neq 0$ if $q\neq 0$.
\end{enumerate}
\end{thm}
\begin{proof}
    We will solve the ODE $Q_\Sigma^h =q$, for a constant $q$.  If $q=0$, using corollary \ref{ZeroQ-K-Cor} we obtain \eqref{K0-revAA}. 
    
    If $q\neq 0$, we define the function $g:\R\to \R$ given by $g(\rho)=f'(\rho)$ and then (\ref{Q-rev-AA}) becomes
    \begin{eqnarray*}
       Q_\Sigma^h&=& \frac{-4 \rho  g'(\rho )-8 g(\rho )^2+8 g(\rho ) \left(\left(\rho
   ^3+\rho \right) g'(\rho )+1\right)-2}{\left(4 \left(\rho ^2+1\right)
   g(\rho )^2-4 g(\rho )+1\right)^{3/2}}.
    \end{eqnarray*}
    Next, by using the substitution \begin{eqnarray*}
        v(\rho) &=& \frac{1}{\sqrt{4(\rho^2+1)g(\rho)^2-4g(\rho)+1}},
    \end{eqnarray*}
     the ODE $Q_\Sigma^h =q$ becomes
    \begin{eqnarray*}
        \rho v'(\rho) +v(\rho) &=& -\frac{q}{2} \\
        (\rho v)' &=& -\frac{q}{2} \\
        v(\rho) &=& -\frac{q}{2}+\frac{c_1}{\rho }.
    \end{eqnarray*}
 Thus \begin{eqnarray*}
        f'(\rho) &=& \frac{v\pm \sqrt{\rho ^2(1-v^2)+1}}{2 v (\rho ^2 +1 )}.
    \end{eqnarray*}
\end{proof}
 \begin{ex}\label{cc-flask}{\bf CC-sphere and the flask.}
     This is the counterpart of Example \ref{cc-bubble}; however, neither of the two objects that we describe below is a surface of revolution. Details will appear in the forthcoming paper \cite{BPST}.  
     
     The projections of the CC-geodesics that join the origin $O=(0,1,0)$ of the affine-additive group and an arbitrary point $p = (a,\lambda, t)$ are solutions of the isoperimetric problem in the hyperbolic plane, that is, they are lifts of hyperbolic circles of the form
     \begin{eqnarray}
         \label{AA-geod}
         \gamma_{k,\phi}(\tau)&=&\left(   k\tau-\frac{\sqrt{k^2+1}}{k}\arctan\left(\frac{k}{\sqrt{k^2+1}+1}\tan(k\tau)\right),\right. \\
      \notag   && \left. \frac{1}{k^2}(\sqrt{k^2+1}+\cos\phi)(\sqrt{k^2+1}+\cos(2 k\tau)),\right.\\
         \notag&&\left.\frac{1}{k}\sin\phi+\frac{1}{k^2}(\sqrt{k^2+1}+\cos\phi)\sin(2k\tau) \right),
     \end{eqnarray}
      where $k\in\R$,  $\phi\in [0,2\pi] $ and $\tau$ lies in any interval of length $2\pi/|k|$. Note also that $\coth r=\sqrt{k^2+1}/k$ is the curvature of the hyperbolic circle. For a point $p$ such that $\zeta=\lambda+it\neq 0$, $\gamma_k,\phi$ is unique; when $k=0$. the curve $\gamma_{k,\phi}$ is a hyperbolic line emanating from $O$. 
      
      Setting $\tau=R>0$ in (\ref{AA-geod}), the {\it Carnot-Carath\'eodory sphere $S_{cc}^R=S_{cc}(O,R)$ centred at $O$ and of radius $R$} is given by the surface patch
      \begin{eqnarray*}
         \sigma(k,\phi)&=&\left(kR-\frac{\sqrt{k^2+1}}{k}\arctan\left(\frac{k}{\sqrt{k^2+1}+1}\tan(k R)\right) , \right. \\
       \notag  && \left. \frac{1}{k^2}(\sqrt{k^2+1}+\cos\phi)(\sqrt{k^2+1}+\cos(2 kR)),\right.\\
         \notag&&\left.\frac{1}{k}\sin\phi+\frac{1}{k^2}(\sqrt{k^2+1}+\cos\phi)\sin(2k R) \right),
     \end{eqnarray*}
     with $(k,\phi)\in(-\pi/\sinh R,\pi/\sinh R)\times(0,2\pi)$.
     By eliminating $\phi$ using $\cos^2\phi+\sin^2\phi=1$ we obtain:
$$
f(k,\lambda,t)=(\lambda^2-t^2-1)\cos(2kR)+2\lambda t\sin(2kR)-\sqrt{k^2+1}\left((\lambda-1)^2+t^2\right)=0.
$$
We also have that
$$
a=g(k)=kR-\frac{\sqrt{k^2+1}}{k}\arctan\left(\frac{k}{\sqrt{k^2+1}+1}\tan(k R)\right),
$$
and the equation $u(a,\lambda,t)=f(g^{-1}(a),\lambda,t )=0$ defines $S_{cc}^R$ as a hypersurface in $\Aa.$

We next set $s=k\tau$ and $\sinh R=1/|k|$ in \eqref{AA-geod} to obtain the surface patch
\begin{eqnarray*}
   \sigma_R(s,\phi)&=&\left(s-\cosh R\arctan\left(e^{-R}\tan s\right),\right.\\
   &&\left.(\cosh R+\cos\phi\sinh R)(\cosh R+\cos(2s)\sinh R)\right.\\
   &&\left. \sinh R(\sin\phi+(\cosh R+\cos\phi\sinh R)\sin(2s))\right),
\end{eqnarray*}
$(s,\phi)\in(-\pi ,\pi )\times (0,2\pi)$. This defines the {\it flask} $\mathcal{F}(O,R)$; note that it is foliated by horizontal curves $\gamma_\phi(s)=\sigma_R(s,\phi)$ which are lifts of hyperbolic circles of radius $R$. With the aid of Proposition \ref{k-H} we find that the horizontal mean curvature of the flask is constant and equal to the curvature of a hyperbolic circle of radius $R$; that is, $H_{\mathcal{F}(R)}^h=\coth R.$
\begin{rem}
    The expressions for the horizontal Gauss curvature and symplectic distortion for both $S_{cc}^R$ and ${\mathcal F}(O, R)$ are not constant.
\end{rem}
\begin{figure}[ht!]
\centering
\includegraphics[width=1\textwidth]{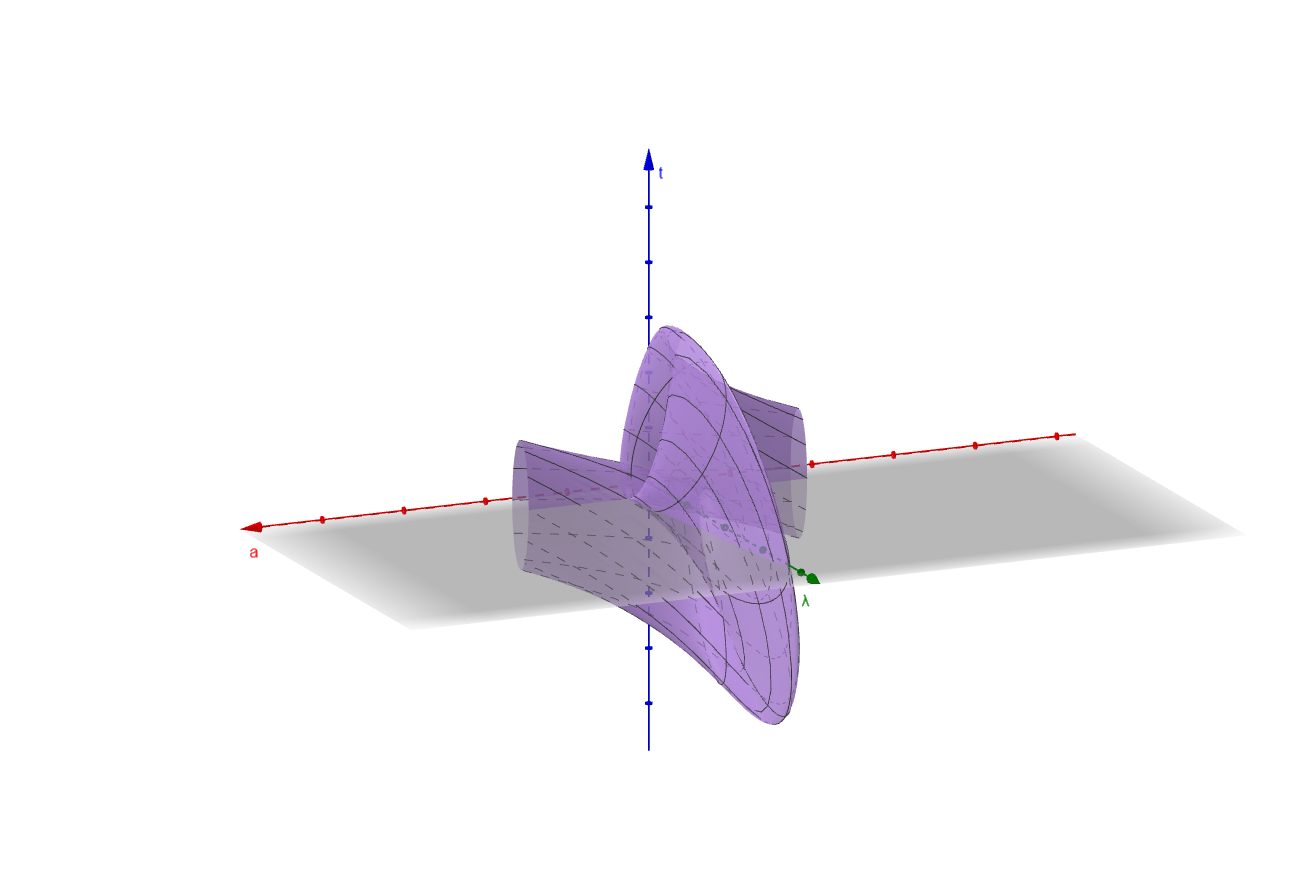}
\caption{Plot of  $\mathcal{F}(O,1)$.}
\end{figure}
 \end{ex}
\newpage

\bibliographystyle{plain}  
\bibliography{Library.bib}    

@article {BPST,
    AUTHOR = {E.~Bubani and I.D.~Platis and G.~Simantiras and D.~Tsolis},
     TITLE = {Sub-{R}iemannian geodesics in the affine-additive groups},
   JOURNAL = {Preprint},
  FJOURNAL = {},
    VOLUME = {Forthcoming},
      YEAR = {2026},
    NUMBER = {},
     PAGES = {},
      ISSN = {},
   MRCLASS = {},
  MRNUMBER = {},
MRREVIEWER = {},
       DOI = {},
       URL = {},
}

@article {B,
    AUTHOR = {Z.M.~Balogh},
     TITLE = {Hausdorff dimension distribution of quasiconformal mappings on
              the {H}eisenberg group},
   JOURNAL = {J. Anal. Math.},
  FJOURNAL = {Journal d'Analyse Math\'{e}matique},
    VOLUME = {83},
      YEAR = {2001},
     PAGES = {289--312},
      ISSN = {0021-7670,1565-8538},
   MRCLASS = {30C65},
  MRNUMBER = {1828495},
MRREVIEWER = {Alexander\ Vasil\cprime ev},
       DOI = {10.1007/BF02790265},
       URL = {https://doi.org/10.1007/BF02790265},
}

@incollection {P,
    AUTHOR = {I.D.~Platis},
     TITLE = {Quasiconformal mappings in the hyperbolic {H}eisenberg group
              and a lifting theorem},
 BOOKTITLE = {Lecture {N}otes of {S}eminario {I}nterdisciplinare di
              {M}atematica. {V}ol. {XV}},
    SERIES = {Lect. Notes Semin. Interdiscip. Mat.},
    VOLUME = {15},
     PAGES = {75--105},
 PUBLISHER = {Semin. Interdiscip. Mat. (S.I.M.), Potenza},
      YEAR = {2020},
      ISBN = {978-88-97478-26-3},
   MRCLASS = {30C65},
  MRNUMBER = {4395000},
MRREVIEWER = {Ilmari\ Kangasniemi},
}

@article {T,
    AUTHOR = {R.~Timsit},
     TITLE = {Geometric construction of quasiconformal mappings in the
              {H}eisenberg group},
   JOURNAL = {Conform. Geom. Dyn.},
  FJOURNAL = {Conformal Geometry and Dynamics. An Electronic Journal of the
              American Mathematical Society},
    VOLUME = {22},
      YEAR = {2018},
     PAGES = {99--140},
      ISSN = {1088-4173},
   MRCLASS = {30C65 (30C75 30L10 32V20)},
  MRNUMBER = {3845546},
MRREVIEWER = {Daniel\ Meyer},
       DOI = {10.1090/ecgd/323},
       URL = {https://doi.org/10.1090/ecgd/323},
}

@book {H,
    AUTHOR = {J.~Heinonen},
     TITLE = {Lectures on analysis on metric spaces},
    SERIES = {Universitext},
 PUBLISHER = {Springer-Verlag, New York},
      YEAR = {2001},
     PAGES = {x+140},
      ISBN = {0-387-95104-0},
   MRCLASS = {30C65 (28A75 28A78 46E35)},
  MRNUMBER = {1800917},
MRREVIEWER = {Christopher\ Bishop},
       DOI = {10.1007/978-1-4613-0131-8},
       URL = {https://doi.org/10.1007/978-1-4613-0131-8},
}

@book {A,
    AUTHOR = {T.~Aubin},
     TITLE = {A course in differential geometry},
    SERIES = {Graduate Studies in Mathematics},
    VOLUME = {27},
 PUBLISHER = {American Mathematical Society, Providence, RI},
      YEAR = {2001},
     PAGES = {xii+184},
      ISBN = {0-8218-2709-X},
   MRCLASS = {53-01 (58-01)},
  MRNUMBER = {1799532},
MRREVIEWER = {Liviu\ I.\ Nicolaescu},
       DOI = {10.1090/gsm/027},
       URL = {https://doi.org/10.1090/gsm/027},
}

@article {W,
    AUTHOR = {M.~Williams},
     TITLE = {Geometric and analytic quasiconformality in metric measure
              spaces},
   JOURNAL = {Proc. Amer. Math. Soc.},
  FJOURNAL = {Proceedings of the American Mathematical Society},
    VOLUME = {140},
      YEAR = {2012},
    NUMBER = {4},
     PAGES = {1251--1266},
      ISSN = {0002-9939,1088-6826},
   MRCLASS = {30L10 (30C65 46E35)},
  MRNUMBER = {2869110},
MRREVIEWER = {Leonid\ V.\ Kovalev},
       DOI = {10.1090/S0002-9939-2011-11035-9},
       URL = {https://doi.org/10.1090/S0002-9939-2011-11035-9},
}

@article {Z,
    AUTHOR = {V.A.~Zorich},
     TITLE = {Asymptotic geometry and conformal types of
              {C}arnot-{C}arath\'eodory spaces},
   JOURNAL = {Geom. Funct. Anal.},
  FJOURNAL = {Geometric and Functional Analysis},
    VOLUME = {9},
      YEAR = {1999},
    NUMBER = {2},
     PAGES = {393--411},
      ISSN = {1016-443X,1420-8970},
   MRCLASS = {53C17},
  MRNUMBER = {1692466},
       DOI = {10.1007/s000390050092},
       URL = {https://doi.org/10.1007/s000390050092},
}

@Book{CDPT,
 Author = {L.~Capogna and D.~Danielli and S.D.~Pauls and J.T.~Tyson},
 Title = {An introduction to the {Heisenberg} group and the sub-{Riemannian} isoperimetric problem},
 FSeries = {Progress in Mathematics},
 Series = {Prog. Math.},
 ISSN = {0743-1643},
 Volume = {259},
 ISBN = {978-3-7643-8132-5},
 Year = {2007},
 Publisher = {Basel: Birkh{\"a}user},
 Language = {English},
 zbMATH = {5083840},
 Zbl = {1138.53003}
}

@Misc{G,
 Author = {H.~Gr{\"o}tzsch},
 Title = {{\"U}ber einige {Extremalprobleme} der konformen {Abbildung}. {I}, {II}.},
 Year = {1928},
 Language = {German},
 HowPublished = {Berichte {Leipzig} 80; 367-376, 497-502 (1928).},
 zbMATH = {2576847},
 JFM = {54.0378.01}
}

@Book{Mont,
 Author = {R.~Montgomery},
 Title = {A tour of subriemannian geometries, their geodesics and applications},
 FSeries = {Mathematical Surveys and Monographs},
 Series = {Math. Surv. Monogr.},
 ISSN = {0076-5376},
 Volume = {91},
 ISBN = {0-8218-1391-9},
 Year = {2002},
 Publisher = {Providence, RI: American Mathematical Society (AMS)},
 Language = {English},
 zbMATH = {1731778},
 Zbl = {1044.53022}
}

@book {Lee-Intro,
    AUTHOR = {J.M.~Lee},
     TITLE = {Introduction to smooth manifolds},
    SERIES = {Graduate Texts in Mathematics},
    VOLUME = {218},
   EDITION = {Second},
 PUBLISHER = {Springer, New York},
      YEAR = {2013},
     PAGES = {xvi+708},
      ISBN = {978-1-4419-9981-8},
   MRCLASS = {58-01 (53-01 57-01)},
  MRNUMBER = {2954043},
}

@Book{Chen-Lin-Biharmonic,
 Author = {B.-Y.~Chen and Y.-L.~Ou},
 Title = {Biharmonic submanifolds and biharmonic maps in {Riemannian} geometry},
 ISBN = {978-981-12-1237-6; 978-981-12-1239-0},
 Year = {2020},
 Publisher = {Hackensack, NJ: World Scientific},
 Language = {English},
 DOI = {10.1142/11610},
 zbMATH = {7184925},
 Zbl = {1455.53002},
}

@article {H-ISR,
    AUTHOR = {R.K.~Hladky},
     TITLE = {Isometries of complemented sub-{R}iemannian manifolds},
   JOURNAL = {Adv. Geom.},
  FJOURNAL = {Advances in Geometry},
    VOLUME = {14},
      YEAR = {2014},
    NUMBER = {2},
     PAGES = {319--352},
      ISSN = {1615-715X,1615-7168},
   MRCLASS = {53C17},
  MRNUMBER = {3263430},
MRREVIEWER = {Erlend\ Grong},
       DOI = {10.1515/advgeom-2014-0015},
       URL = {https://doi.org/10.1515/advgeom-2014-0015},
}

@article {B-Thesis,
    AUTHOR = {E.~Bubani},
     TITLE = {Hyperbolicity and quasiconformal maps on the affine-additive group},
   JOURNAL = {Doctoral Dissertation},
  FJOURNAL = {},
    VOLUME = {\url{https://boristheses.unibe.ch/6619/}},
      YEAR = {Universit\"at Bern, 2025},
    NUMBER = {},
     PAGES = {},
      ISSN = {},
   MRCLASS = {},
  MRNUMBER = {},
MRREVIEWER = {},
       DOI = {},
       URL = {},
}

@article {RR-Rot.inv.,
    AUTHOR = {M.~Ritor\'e and C.~Rosales},
     TITLE = {Rotationally invariant hypersurfaces with constant mean
              curvature in the {H}eisenberg group {$\mathbb{H}^n$}},
   JOURNAL = {J. Geom. Anal.},
  FJOURNAL = {The Journal of Geometric Analysis},
    VOLUME = {16},
      YEAR = {2006},
    NUMBER = {4},
     PAGES = {703--720},
      ISSN = {1050-6926,1559-002X},
   MRCLASS = {53C17 (49Q20 53C42)},
  MRNUMBER = {2271950},
MRREVIEWER = {Constantin\ Vernicos},
       DOI = {10.1007/BF02922137},
       URL = {https://doi.org/10.1007/BF02922137},
}

@article{KD-17,
 author = {V.~Kivioja and E.~Le Donne},
 title = {Isometries of nilpotent metric groups},
 fjournal = {Journal de l'{\'E}cole Polytechnique -- Math{\'e}matiques},
 journal = {J. {\'E}c. Polytech., Math.},
 issn = {2429-7100},
 volume = {4},
 pages = {473--482},
 year = {2017},
 language = {English},
 doi = {10.5802/jep.48},
 zbMATH = {6754333},
 Zbl = {1369.22006}
}

@book{D-MLG,
 author = {E.~Le Donne},
 title = {Metric {Lie} groups. {Carnot}-{Carath{\'e}odory} spaces from the homogeneous viewpoint},
 fseries = {Graduate Texts in Mathematics},
 series = {Grad. Texts Math.},
 issn = {0072-5285},
 volume = {306},
 isbn = {978-3-031-98831-8; 978-3-031-98834-9; 978-3-031-98832-5},
 year = {2025},
 publisher = {Cham: Springer},
 language = {English},
 doi = {10.1007/978-3-031-98832-5},
 zbMATH = {8074076}
}

@article {DO-I,
    AUTHOR = {E.~Le Donne and A.~Ottazzi},
     TITLE = {Isometries of {C}arnot groups and sub-{F}insler homogeneous
              manifolds},
   JOURNAL = {J. Geom. Anal.},
  FJOURNAL = {Journal of Geometric Analysis},
    VOLUME = {26},
      YEAR = {2016},
    NUMBER = {1},
     PAGES = {330--345},
      ISSN = {1050-6926,1559-002X},
   MRCLASS = {53C17 (53C60 58D05)},
  MRNUMBER = {3441517},
MRREVIEWER = {Davide\ Vittone},
       DOI = {10.1007/s12220-014-9552-8},
       URL = {https://doi.org/10.1007/s12220-014-9552-8},
}

@InCollection{Gromov,
 Author = {M.~Gromov},
 Title = {Carnot-{Carath{\'e}odory} spaces seen from within},
 BookTitle = {Sub-Riemannian geometry. Proceedings of the satellite meeting of the first European congress of mathematics `Journ\'ees nonholonomes: g\'eom\'etrie sous-riemannienne, th\'eorie du contr\^ole, robotique', Paris, France, June 30--July 1, 1992},
 ISBN = {3-7643-5476-3},
 Pages = {79--323},
 Year = {1996},
 Publisher = {Basel: Birkh{\"a}user},
 Language = {English},
 zbMATH = {953396},
 Zbl = {0864.53025}
}

@article {BTV,
    AUTHOR = {Z.M.~Balogh and J.T.~Tyson and E.~Vecchi},
     TITLE = {Intrinsic curvature of curves and surfaces and a
              {G}auss-{B}onnet theorem in the {H}eisenberg group},
   JOURNAL = {Math. Z.},
  FJOURNAL = {Mathematische Zeitschrift},
    VOLUME = {287},
      YEAR = {2017},
    NUMBER = {1-2},
     PAGES = {1--38},
      ISSN = {0025-5874,1432-1823},
   MRCLASS = {53C17 (52A39 53A35)},
  MRNUMBER = {3694666},
MRREVIEWER = {Davide\ Vittone},
       DOI = {10.1007/s00209-016-1815-6},
       URL = {https://doi.org/10.1007/s00209-016-1815-6},
}

@Book{DC,
 Author = {M.P.~do Carmo},
 Title = {Riemannian geometry. {Translated} from the {Portuguese} by {Francis} {Flaherty}},
 ISBN = {0-8176-3490-8},
 Year = {1992},
 Publisher = {Boston, MA etc.: Birkh{\"a}user},
 Language = {English},
 zbMATH = {52737},
 Zbl = {0752.53001}
}

@misc{BBP-GaussBonnet,
 author = {D.~Barilari and E.~Bellini and A.~Pinamonti},
 title = {Curvature measures and the sub-{Riemannian} {Gauss}-{Bonnet} theorem},
 year = {2025},
 howpublished = {Preprint, {arXiv}:2509.26460 [math.{DG}] (2025)},
 url = {https://arxiv.org/abs/2509.26460},
 arXiv = {arXiv:2509.26460}
}

@article{DV-GaussBonnet,
 author = {M. M.~Diniz and J. M. M.~Veloso},
 title = {Gauss-Bonnet theorem in sub-{Riemannian} {Heisenberg} space ${H}^1$ },
 fjournal = {Journal of Dynamical and Control Systems},
 journal = {J. Dyn. Control Syst.},
 issn = {1079-2724},
 volume = {22},
 number = {4},
 pages = {807--820},
 year = {2016},
 language = {English},
 doi = {10.1007/s10883-016-9338-3},
 zbMATH = {6646291},
 Zbl = {1362.53040}
}

@article{GHVM-GaussBonnet,
 author = {E.~Grong and J.~Hidalgo and S.~Vega-Molino},
 title = {A sub-{Riemannian} {Gauss}-{Bonnet} theorem for surfaces in contact manifolds},
 fjournal = {The Journal of Geometric Analysis},
 journal = {J. Geom. Anal.},
 issn = {1050-6926},
 volume = {35},
 number = {8},
 pages = {33},
 note = {Id/No 243},
 year = {2025},
 language = {English},
 doi = {10.1007/s12220-025-02076-3},
 zbMATH = {8072179}
}

@article{V-GBgeneral,
 author = {J. M. M.~Veloso},
 title = {Gauss-Bonnet theorems for surfaces in sub-{Riemannian} three-dimensional manifolds},
 fjournal = {Journal of Dynamical and Control Systems},
 journal = {J. Dyn. Control Syst.},
 issn = {1079-2724},
 volume = {29},
 number = {3},
 pages = {1055--1076},
 year = {2023},
 language = {English},
 doi = {10.1007/s10883-022-09626-w},
 zbMATH = {7753816},
 Zbl = {1526.53029}
}

@article{Vel-RSCC1,
 author = {J. M. M.~Veloso},
 title = {Riemannian approximation scheme in sub-{Riemannian} {Heisenberg} space {{\(\mathbb{H}^1\)}} and rotation surfaces of constant {Gaussian} curvature},
 fjournal = {Matem{\'a}tica Contempor{\^a}nea},
 journal = {Mat. Contemp.},
 issn = {0103-9059},
 volume = {50},
 pages = {302--320},
 year = {2022},
 language = {English},
 doi = {10.21711/231766362022/rmc5011},
 zbMATH = {7842501},
 Zbl = {1549.53107}
}

@article{Vel-RSCC2,
 author = {J. M. M.~Veloso},
 title = {Rotation surfaces of constant {Gaussian} curvature and mean curvature in sub-{Riemannian} {Heisenberg} space {{\(\mathbb{H}^1\)}}},
 fjournal = {Journal of Geometry},
 journal = {J. Geom.},
 issn = {0047-2468},
 volume = {113},
 number = {2},
 pages = {19},
 note = {Id/No 38},
 year = {2022},
 language = {English},
 doi = {10.1007/s00022-022-00652-4},
 zbMATH = {7565395},
 Zbl = {1505.53046}
}

@article{Massey,
 author = {W. S.~Massey},
 title = {Surfaces of {Gaussian} curvature zero in {Euclidean} 3-space},
 fjournal = {T{\^o}hoku Mathematical Journal. Second Series},
 journal = {T{\^o}hoku Math. J. (2)},
 issn = {0040-8735},
 volume = {14},
 pages = {73--79},
 year = {1962},
 language = {English},
 doi = {10.2748/tmj/1178244205},
 zbMATH = {3186457},
 Zbl = {0114.36903}
}

@article{P-SRS,
 author = {I.D.~Platis},
 title = {Straight ruled surfaces in the {Heisenberg} group},
 fjournal = {Journal of Geometry},
 journal = {J. Geom.},
 issn = {0047-2468},
 volume = {105},
 number = {1},
 pages = {119--138},
 year = {2014},
 language = {English},
 doi = {10.1007/s00022-013-0199-6},
 zbMATH = {6348441},
 Zbl = {1298.53053}
}

@article{M-76,
title = {Curvatures of left invariant metrics on lie groups},
journal = {Advances in Mathematics},
volume = {21},
number = {3},
pages = {293-329},
year = {1976},
issn = {0001-8708},
doi = {https://doi.org/10.1016/S0001-8708(76)80002-3},
url = {https://www.sciencedirect.com/science/article/pii/S0001870876800023},
author = {J.~Milnor}
}
\Addresses

\end{document}